\documentclass[12pt,reqno]{amsart}
\usepackage[foot]{amsaddr}
\usepackage{amsmath}
\usepackage{amssymb}
\usepackage{verbatim}
\usepackage{bbm}
\usepackage{thmtools}
\usepackage{graphicx}
\usepackage[font=footnotesize]{caption}
\usepackage{subcaption}
\usepackage{epstopdf}
\usepackage{multirow}
\usepackage{url}
\usepackage{makecell}
\usepackage{tabularx}
\usepackage{longtable}
\usepackage{graphicx}
\usepackage[percent]{overpic}
\usepackage{tikz}
\usepackage{pgfplots}
\usepackage{xcolor}
\usepackage{pgfplots}
\usepackage{pgfplotstable}
\usepackage{wavelet}
\allowdisplaybreaks
\usepackage[toc]{appendix}
\usepackage[capitalise]{cleveref}
\usepackage{color}
\usepackage{cases}
\allowdisplaybreaks

\usepackage[top=0.5in,bottom=0.47in,left=0.7in,right=0.7in]{geometry}

\newtheorem{theorem}{Theorem}
\newtheorem{prop}[theorem]{Proposition}
\newtheorem{example}{Example}

\newcommand{\nv}{\vec{n}} %normal vector

\def\XXint#1#2#3{{\setbox0=\hbox{$#1{#2#3}{\int}$ }
		\vcenter{\hbox{$#2#3$ }}\kern-.6\wd0}}

\renewcommand{\ge}{\geqslant}
\renewcommand{\le}{\leqslant}

\newcommand{\blf}{B} %bilinear form use either a or b.

\numberwithin{equation}{section}
\begin{document}

%\title[Galerkin Scheme Using Wavelets for 1D Elliptic Interface Problems]{Galerkin Scheme Using Biorthogonal Wavelets on Intervals for 1D Elliptic Interface Problems}

\title[Wavelet Galerkin for 1D Elliptic Interface Problems]{
Arbitrarily High-Order Convergence of Wavelet-based Galerkin Scheme for 1D Elliptic Interface Problems
%Wavelet-based Galerkin Scheme with Arbitrarily High-Order Convergence for 1D Elliptic Interface Problems
}

	\author{Bin Han}
	\address{Department of Mathematical and Statistical Sciences, University of Alberta, Edmonton, Alberta, Canada T6G 2N8.}
	\email{bhan@ualberta.ca, mmichell@ualberta.ca}
	
	\author{Michelle Michelle}
	%\address{Department of Mathematics, Purdue University, West Lafayette, IN, USA 47907.}
	%\email{mmichell@purdue.edu}
	
	\thanks{Research supported in part by
		Natural Sciences and Engineering Research Council (NSERC) of Canada under grants RGPIN-2024-04991, RGPIN-2026-05814, and DGECR-2026-00371, the University of Alberta startup grant, and the Digital Research Alliance of Canada.}

%\thanks{\footnotesize{\textbf{Declaration: The authors certify that the general content of the manuscript, in whole or in part, is not submitted, accepted, or published elsewhere, including conference proceedings.}}}

	\makeatletter \@addtoreset{equation}{section} \makeatother

	% The key contribution is the rigourous convergence analysis. Insight how the proof works in higher dimensions. 
	% What key techniques do you use to prove it. 

	\begin{abstract}
	The solution $u$ of an elliptic interface problem in a domain $\Omega$ is often smooth away from the interface $\Gamma\subset \Omega$, but its gradient is discontinuous across $\Gamma$. Consequently, $u$ has low global regularity and generally does not belong to $H^{3/2}(\Omega)$. This paper studies 1D elliptic interface problems using wavelet methods. We propose a Galerkin method based on compactly supported biorthogonal wavelet bases on bounded intervals with approximation order $m$, for any integer $m \ge 2$. Our approach involves incorporating wavelet basis functions from higher scale levels to capture the singularity in the neighbourhood of the interface $\Gamma$. A principal contribution of this paper is a rigorous convergence analysis establishing the optimal orders $m-1$ in the $H^1(\Omega)$-norm and $m$ in the $L^2(\Omega)$-norm. The proof combines careful decay estimates for dual wavelet coefficients associated with wavelets supported away from the interface $\Gamma$ and with those supported in its neighborhood, wavelet characterizations of Sobolev spaces, and standard convergence arguments from the finite element method (FEM). Various numerical experiments are provided to verify the theoretical findings, including a special two-dimensional elliptic interface problem with vertical interface lines.  
	 %we rigorously prove that its convergence rates are of order $m-1$ in the $H^1(\Omega)$-norm and order $m$ in the $L^2(\Omega)$-norm, which are optimal with respect to the scheme's approximation order $m$. %The results in this paper both complement and sharply contrast our findings in \cite{HM24}, where we consider a similar wavelet-based method for solving $d$-dimensional elliptic interface problems with $d\ge 2$.
	\end{abstract}

	\maketitle

	\pagenumbering{arabic}
	
	% To provide the intuition in 1D how we deal with the singularity. N_J/N_{J-1} \rightarrow 2.
	% Shan Zhao and Zhilin Li's papers
	
	\section{Introduction}
	Let $\Omega\subset \R^d$ be a bounded domain. Consider a smooth $(d-1)$-dimensional interface $\Gamma$ inside $\Omega$ such that it splits the whole domain $\Omega$ into two subdomains $\Omega_+$ and $\Omega_-$. Let $\nv$ stand for the unit normal on $\Gamma$ pointing to the subdomain $\Omega_+$.
The $d$-dimensional elliptic interface problem on $\Omega$ considers
%can be formulated as
	%
	\begin{equation} \label{model:0}
	\begin{cases}
	 -\nabla\cdot(a\nabla u) = f & \text{in} \quad \Omega\bs \Gamma,\\
	 \lb u \rb =g \quad \mbox{and} \quad \lb a \nabla u \cdot \nv \rb=g_{\Gamma} & \mbox{on} \quad \Gamma,\\
	u=g_b, & \mbox{on} \quad \partial \Omega,
	\end{cases}
	\end{equation}
	where the variable diffusion coefficient $a \in L^{\infty}(\Omega)$ satisfies $\text{ess-inf}_{x\in \Omega}(a(x)) > 0$, the source term $f\in L^{2}(\Omega)$,  the jump conditions $g \in H^{1/2}(\Gamma)$ and $g_{\Gamma} \in H^{-1/2}(\Gamma)$,
	and the boundary condition $g_b\in H^{1/2}(\partial \Omega)$.
Throughout the paper we use the notations $u_\pm:=u\chi_{\Omega_\pm}$, $a_\pm:=a \chi_{\Omega_\pm}$, and $f_\pm:=f\chi_{\Omega_\pm}$.
The first jump function  $\lb u \rb$ on $\Gamma$ for the solution $u$ to \eqref{model:0} is defined by
\[
\lb u \rb(x):=\lim_{h\to 0^+} u_+(x+h \nv)-\lim_{h\to 0^+}u_-(x-h \nv),\qquad x\in \Gamma.
\]
The second jump condition
$\lb a \nabla u \cdot \nv \rb$ on $\Gamma$ for the flux function $a \nabla u \cdot \nv$ is defined similarly.

	If the first jump condition $g$ is not identically zero, then the solution $u$ to \eqref{model:0} is discontinuous across $\Gamma$.
	For the finite element method (FEM), one typically constructs auxiliary functions to reduce the problem to one with $g=0$ on $\Gamma$ and $g_b=0$ on $\partial \Omega$ in the model problem \eqref{model:0}. Consequently, such an elliptic interface problem with $g=0$ on $\Gamma$ and $g_b=0$ on $\partial \Omega$ can be equivalently rewritten as
	\begin{equation} \label{model:0:delta}
		 -\nabla\cdot(a\nabla u) = f - g_{\Gamma} \delta_\Gamma \quad \text{in} \quad \Omega
		\quad \mbox{with}\quad
		u=0 \quad \mbox{ on } \partial \Omega,
	\end{equation}
where $g_\Gamma \delta_\Gamma$ is the Dirac distribution along the interface $\Gamma$ with weight $g_\Gamma$.

	In most of the literature on elliptic interface problems, rectangular domains $\Omega$ are often considered, and both the diffusion coefficient $a$ and the source term $f$ are assumed to be smooth in the subdomains $\Omega_+$ and $\Omega_-$. However, both $a$ and $f$ can be discontinuous across the interface $\Gamma$. As a result, even with $g=0$ in the first jump condition, the solution $u$ to the problem \eqref{model:0} has a discontinuous gradient across the interface, which results in $u\notin H^{1.5}(\Omega)$, even though $u_+ \in H^{m}(\Omega_+)$ and $u_- \in H^{m}(\Omega_-)$ for a large integer $m$. Specifically, the gradient $\nabla u$ is typically discontinuous across the interface $\Gamma$ if $a$ is discontinuous across $\Gamma$, $f$ is discontinuous across $\Gamma$ (even if $a=1$), or $g_{\Gamma}$ is not identically zero. Without taking into account the geometry of the interface $\Gamma$, standard discretization methods suffer from low convergence orders, which are usually at most $1/2$ in the $H^1(\Omega)$-norm.
	
	Many studies have been devoted to overcoming this issue by introducing various modifications to the finite element, standard or discontinuous Galerkin, finite difference, and other methods, see \cite{Aetal14, BE18, CZZZ17,FHM24,LAH09,HM24,HLL11,GL19,LL94} and many references therein. The one-dimensional elliptic interface problem is specifically addressed in \cite{CZZZ17,LAH09,HLL11}. For the finite difference method (FDM), the stencils near the interface are modified to take into account the geometry of the interface $\Gamma$ and the two given jump conditions $g, g_\Gamma$ on $\Gamma$ in \eqref{model:0}. The highest order of FDMs for the 2D elliptic interface problem \eqref{model:0}  is probably known to be 6 in the $L^\infty(\Omega)$-norm, which is achieved by using $13$-point stencils near $\Gamma$ and $9$-point compact stencils everywhere else \cite{FHM24}. See \cite{FHM24} for a comprehensive list of references for using FDMs to solve elliptic interface problems \eqref{model:0}. Meanwhile, a key idea in the FEM is to modify the elements near the interface, which often leads to nonconforming/discontinuous elements. Many modified FEMs achieve the first-order convergence in the $H^1(\Omega)$-norm and the second-order convergence in the $L^2(\Omega)$-norm for the 2D elliptic interface problems \eqref{model:0:delta}. See \cite{BE18,HM24,HLL11,GL19,LL94} for a large body of work on the numerical solution of the elliptic interface problems \eqref{model:0:delta} (i.e., \eqref{model:0} with $g=0$ on $\Gamma$ and $g_b=0$ on $\partial \Omega$) using the FEM. The rigorous convergence analysis of modified FEMs for the elliptic interface problems \eqref{model:0:delta} has been studied in the literature, whereas such analysis for modified FDMs remains scarce. To date, establishing high-order modified FDMs (with convergence rates exceeding $6$th-order in the $L^\infty(\Omega)$-norm) and FEMs (with convergence rates exceeding $2$nd-order in the $L^2(\Omega)$-norm) remains a challenging and important task.

	Over the years, wavelets and framelets have been widely used in signal and image processing, data analysis, and the numerical solution of differential equations (e.g., see \cite{hanbook} and references therein). In this paper, we are interested in the wavelet-based Galerkin method for discretizing the model problem \eqref{model:0:delta}. 
	%Wavelet-based methods have been used to solve PDEs numerically for decades. 
	Some of their advantages are that they are conforming meshless methods, and due to the well-behaved nature of the underlying basis, the condition numbers of the coefficient matrices are uniformly bounded, regardless of their size. The authors of \cite{cdd01} introduced the influential adaptive wavelet method for the numerical solution of general PDEs. See \cite{cdd01,dku99,HM21,HM24,HM25} and references therein for the construction of wavelets on the bounded intervals and their applications to numerical PDEs. The results in \cite[Section~3.1 and Proposition~3.2]{cdd01} state that if the solution $u$ belongs to the Besov space $B^{sd+1}_\tau (L^\tau(\Omega))$ with $\frac{1}{\tau}=s+\frac{1}{2}$, then the adaptive wavelet method in \cite{cdd01}, applied to the elliptic interface problem \eqref{model:0:delta}, achieves a convergence rate of $\bo(n^{-s})$ in the $H^1(\Omega)$-norm for adaptively selected $n$ wavelet elements (see \cite[Proposition~3.2]{cdd01} for details) and particularly $\bo(h^{sd})$ for $n=\bo(h^{-d})$, where $h$ is the mesh size. Because the solution $u$ of the elliptic interface problem \eqref{model:0:delta} has a discontinuous gradient but $u$ still has piecewise regularity, the solution $u$ belongs to $H^{1.5-\delta}(\Omega)$ for every $\delta >0$. Note that $H^{3/2-\delta}(\Omega) = B^{3/2-\delta}_2(L^{2}(\Omega)) \hookrightarrow B^{1+sd}_{\tau}(L^{\tau}(\Omega))$, where $sd+1 < 1.5-\delta$ (i.e., $sd<0.5$) and $\tfrac{1}{\tau} = s+ \tfrac{1}{2}$. Thus, the adaptive wavelet method in \cite{cdd01} can only claim a proved convergence rate of $\bo(h^{\gep})$ in the $H^1(\Omega)$-norm, where $\gep<0.5$ is any fixed constant. This naturally motivates us to ask the following question: %particular one-dimensional question:
	
	\textbf{Question:}
	\emph{Consider the following 1D elliptic interface problem (i.e., \eqref{model:0:delta} with $d=1$):
		\begin{equation} \label{model}
		 -(au')' = f - g_{\Gamma} \delta_{\Gamma} \quad \text{in} \quad \Omega:=(0,1)
		 \quad \text{with} \quad
		 u(0) = u(1)=0,
		\end{equation}
	where $\Gamma$ is a finite subset of $\Omega$, $g_{\Gamma}: \Gamma \rightarrow \R$, and functions $a,f$ are smooth enough in $\Omega\setminus \Gamma$.
%$\Omega_+$ and $\Omega_-$ with $\Omega_- := (0,\Gamma)$ and $\Omega_+ := (\Gamma,1)$.
Let $m\ge 2$ be an integer. Assume that %$u_+ \in H^{m}(\Omega_+)$ and  $u_- \in H^{m}(\Omega_-)$,
$u\in H^m(\Omega\setminus \Gamma)$, but $u\notin H^{1.5}(\Omega)$ due to its discontinuous gradient across $\Gamma$.
Moreover, the coefficient $a$ often has sharp contrast across points in $\Gamma$ and some points in $\Gamma$ could be very close to each other causing difficulties even for current advanced FEMs and FDMs for elliptic interface problems.
 Is it possible to design a wavelet method that achieves convergence rates of $m-1$ and $m$ in $H^{1}(\Omega)$ and $L^2(\Omega)$ norms respectively (i.e., optimal convergence rates with respect to the approximation order $m$ of the underlying scheme)?}
	
	\textbf{Answer:} \textit{Yes.}
	
	The main contributions of this paper are threefold:
	\begin{itemize}
	\item[(1)] We rigorously justify the positive answer to the above question with a self-contained proof. Our wavelet Galerkin method relies on a simple yet effective strategy, which essentially involves placing more wavelet basis functions from higher scale levels that touch the interface $\Gamma$ to capture the singularity. The analysis relies heavily on the properties of the compactly supported dual wavelets. More specifically, the proof combines careful decay estimates for dual wavelet coefficients associated with wavelets supported away from the interface $\Gamma$ and with those supported in its neighborhood, wavelet characterizations of Sobolev spaces, and standard convergence arguments from the finite element method (FEM).
	 \item[(2)] This paper aims to complement and contrast our findings in \cite{HM24}, where a similar strategy is used to solve $d$-dimensional elliptic interface problems. The 1D setting warrants a separate treatment due to the possibility of achieving arbitrarily high convergence rates. Furthermore, it is studied independently to provide a clear exposition of the main ideas behind our method and its convergence proof. 
	\item[(3)] We provide some numerical experiments to validate our theoretical findings and demonstrate the performance of our method. In particular, we show that our method can be used for solving a 2D elliptic interface equation with vertical interface lines. The simplicity of these lines allows use to directly use our 1D approach to solve the problem. 
	\end{itemize}
	
	The rest of this paper is organized in three sections. \cref{sec:main} contains the main result in the form of theoretically guaranteed convergence rates. \cref{sec:numerical} contains some numerical results, which support our theoretical results.
The main result stated in \cref{sec:main} is proved in \cref{sec:proof}.

	\section{Main Result Using Biorthogonal Multiwavelets on Bounded Intervals} \label{sec:main}
	We start this section by recalling some basic facts on wavelets, which serve as a basis for our discretization. For a detailed exposition, we refer readers to \cite[Sections 2.1-2.2]{HM24} or \cite[Section 1.1]{HM25}. Starting with a compactly supported biorthogonal wavelet $(\{\tilde{\phi};\tilde{\psi}\},\{\phi;\psi\})$ in $L^{2}(\R)$, we can construct a compactly supported biorthogonal wavelet in $L^{2}(\Omega)$, where $\Omega:=(0,1)$, satisfying homogeneous boundary conditions and desired vanishing moments by using the direct approach in \cite{HM21}. Throughout this paper, we solely rely on our construction procedure in \cite{HM21}, because it is simple to use, generates all possible locally supported boundary basis functions, and is applicable to all compactly supported biorthogonal multiwavelets in $L^{2}(\R)$. Recall that a multiwavelet means $\phi, \psi, \tilde{\phi}, \tilde{\psi}$ are vector-valued functions. The second property enables us to avoid using more boundary basis functions than needed; i.e., we maximally preserve the interior basis functions. For a given approximation order, a multiwavelet generally has shorter support than a (scalar) wavelet and requires fewer boundary basis functions when adapted to a bounded interval, which simplifies implementation. For simplicity, we mostly use the term wavelets to refer to scalar or vector-valued $\phi, \psi, \tilde{\phi}, \tilde{\psi}$ in this paper.
	
	More explicitly, our direct approach in \cite{HM21} yields a locally supported biorthogonal wavelet $(\tilde{\mathcal{B}}_{J_0},\mathcal{B}_{J_0})$ in $L^2(\Omega)$ from a compactly supported biorthogonal wavelet $(\{\tilde{\phi};\tilde{\psi}\},\{\phi;\psi\})$ in $L^{2}(\R)$, where
	\[
	 \mathcal{B}_{J_0}:=\Phi_{J_0}\cup \cup_{j=J_0}^\infty \Psi_j \subseteq L^2(\Omega),
	\qquad \tilde{\mathcal{B}}_{J_0}:=\tilde{\Phi}_{J_0}\cup \cup_{j=J_0}^\infty \tilde{\Psi}_j \subseteq L^2(\Omega),
	\]
	the integer $J_0\in \N$ is the coarsest scale level, and
	\[
	\begin{aligned}
		&\Phi_{J_0} := \{\phi^{L}_{J_0;0}\}\cup
		\{\phi_{J_0;k} \setsp n_{l,\phi} \le k\le 2^{J_0}-n_{h,\phi}\}\cup \{ \phi^{R}_{J_0;2^{J_0}-1}\},\\
		&\Psi_{j} := \{\psi^{L}_{j;0}\}\cup
		\{\psi_{j;k} \setsp n_{l,\psi} \le k\le 2^{j}-n_{h,\psi}\}\cup \{ \psi^{R}_{j;2^{j}-1}\},\quad j\ge J_0,
	\end{aligned}
	\]
	where $f_{j;k}:=2^{j/2}f(2^j\cdot-k)$, $n_{l,\phi}, n_{h,\phi}, n_{l,\psi}, n_{h,\psi}$ are known integers, $\phi^L, \phi^R$ are boundary refinable functions, and $\psi^L, \psi^R$ are boundary wavelets. The boundary basis functions $\phi^L, \phi^R,\psi^L, \psi^R$ are finite subsets of functions in $L^2(\Omega)$. We define $\tilde{\mathcal{B}}_{J_0}$ in the same fashion, except to each element in $\mathcal{B}_{J_0}$, we add $\sim$ for a natural bijection.
Define scaled systems
	\be \label{BH1J0}
 	 \mathcal{B}^{H^1_0(\Omega)}_{J_0} := [2^{-J_0}\Phi_{J_0}] \cup \cup_{j=J_0}^{\infty} [2^{-j}\Psi_{j}]
 \quad \mbox{and}\quad
 \tilde{\mathcal{B}}^{H^{-1}(\Omega)}_{J_0} := [2^{J_0}\tilde{\Phi}_{J_0}] \cup \cup_{j=J_0}^{\infty} [2^{j}\tilde{\Psi}_{j}].
	\ee
By \cite[Theorem 1.1]{HM25}, if $\phi$ belongs to $H^{1}(\R)$ and $\mathcal{B}_{J_0} \subseteq H^{1}_0(\Omega)$, then
 	 $\mathcal{B}^{H^1_0(\Omega)}_{J_0}$
	is a Riesz basis of the Sobolev space $H^1_0(\Omega)$. In fact, \cite[Theorem 1.1]{HM25} is established by proving that  $(\tilde{\mathcal{B}}^{H^{-1}(\Omega)}_{J_0}, \mathcal{B}^{H^{1}_0(\Omega)}_{J_0})$
forms a pair of dual Sobolev spaces $(H^{-1}(\Omega), H^{1}_0(\Omega))$ (see \cite{hanbook} for definition). Consequently,
using a pair of dual Sobolev spaces,
each function $u \in H^{1}_0(\Omega)$ admits the following wavelet representation
	\begin{equation}\label{expr}
	u = \sum_{\eta \in \mathcal{B}^{H^1_0(\Omega)}_{J_0}} c_\eta \eta,
    \end{equation}
	with $c_\eta:=\la u, \tilde{\eta}\ra$, $\tilde{\eta}\in \tilde{\mathcal{B}}^{H^{-1}(\Omega)}_{J_0}$, and the above series converging absolutely in $H^{1}_0(\Omega)$. Additionally, as stated in \cite[Theorem 1.1]{HM25} using a pair of dual Sobolev spaces $(H^{-1}(\Omega), H^{1}_0(\Omega))$, there are positive constants $C_{\mathcal{B},1},C_{\mathcal{B},2}$ such that the coefficients $\{c_\eta\}_{\eta \in \mathcal{B}^{H^1_0(\Omega)}_{J_0}}$ satisfy the following stability property
	\begin{equation} \label{Riesz:stab} %\label{Riesz:L2}
		C_{\mathcal{B},1} \sum_{\eta \in \mathcal{B}^{H^1_0(\Omega)}_{J_0}} |c_{\eta}|^2 \le \Big\| \sum_{\eta \in \mathcal{B}^{H^1_0(\Omega)}_{J_0}} c_{\eta} \eta \Big\|^2_{H^{1}(\Omega)} \le C_{\mathcal{B},2} \sum_{\eta \in \mathcal{B}^{H^1_0(\Omega)}_{J_0}} |c_{\eta}|^2,
	\end{equation}
which is well known as the norm equivalence in the literature (e.g., see \cite{cdd01,dku99} and references therein).

	In the standard wavelet Galerkin framework, we take a finite subset of \eqref{BH1J0} defined below
	\be \label{BH1J0J}
	 \mathcal{B}^{H^1_0(\Omega)}_{J_0,J} := [2^{-J_0}\Phi_{J_0}] \cup \cup_{j=J_0}^{J} [2^{-j} \Psi_{j}],
	\ee
	and use its linear combination as an approximate solution.
However, using this finite subset alone is not sufficient to resolve the singularity at the interface point $\Gamma$.
Note that $\text{span}(\mathcal{B}^{H^1_0(\Omega)}_{J_0,J}) = \text{span}(\Phi_{J+1})$; i.e., $\mathcal{B}^{H^1_0(\Omega)}_{J_0,J}$ spans the same space as the standard (unmodified) FEM.
Note
$\#\mathcal{B}^{H^1_0(\Omega)}_{J_0,J}
=\#\Phi_{J+1}$, where $\#$ stands for the cardinality of the given set.
As discussed in the introduction, this means that one would experience reduced convergence rates. To achieve optimal $H^1(\Omega)$ and $L^2(\Omega)$ convergence rates, we need to include wavelet basis functions at higher scale levels that touch the interface point $\Gamma$. Hence, given that $\vmo(\tilde{\psi}) = m$ (i.e., $\int_{\mathbb{R}}x^{j} \tilde{\psi}(x)dx =0$ for all $j=0,\ldots,m-1$), we enrich our finite subset as follows
	\be \label{BSH1J0J}
	 \mathcal{B}^{S,H^1_0(\Omega)}_{J_0,J}
	 :=\mathcal{B}^{H^1_0(\Omega)}_{J_0,J} \cup \cup_{j=J+1}^{(2m-2)J-1} [2^{-j}\mathcal{S}_j],
	\quad
	\mathcal{S}_{j} := \{\eta \in \Psi_{j} : \Gamma \cap \text{Supp}(\tilde{\eta}) \neq \emptyset  \; \text{and} \; \tilde{\eta} \in \tilde{\Psi}_{j}\},
	\ee
	where $\mathcal{B}^{H^1_0(\Omega)}_{J_0,J}$ is defined in \eqref{BH1J0J}, and $\text{Supp}(\tilde{\eta})$ is the smallest closed interval such that $\tilde{\eta}$ vanishes outside it. See \cref{Sj} for an illustration of the basis functions we add to $\mathcal{B}^{H^1_0(\Omega)}_{J_0,J}$ to form $\mathcal{B}^{S,H^1_0(\Omega)}_{J_0,J}$. The coefficients of the wavelet basis functions are obtained by solving the following linear system
	\begin{equation} \label{BuJv}
		B(u_J,v) := \la a u_J', v' \ra  = \la f, v\ra  - \la g_{\Gamma} \delta_{\Gamma}, v \ra , \quad \forall v \in \mathcal{B}^{S,H^{1}_0(\Omega)}_{J_0,J}.
	\end{equation}
	
	\begin{figure}
		\tikzset{every picture/.style={line width=0.75pt}} %set default line width to 0.75pt        
		\begin{tikzpicture}[x=0.75pt,y=0.75pt,yscale=-1,xscale=1]
			
			%Straight Lines [id:da09133226519803916] 
			\draw    (110.2,260.44) -- (510.2,261.44) ;
			%Straight Lines [id:da37288228997887185] 
			\draw    (347,254.64) -- (347,267.64) ;
			%Straight Lines [id:da8108248586017661] 
			\draw    (321.33,240.41) -- (328.58,251.72) ;
			%Straight Lines [id:da5700050120184311] 
			\draw    (328.58,251.72) -- (335.83,210.65) ;
			%Straight Lines [id:da879217878359737] 
			\draw    (343.08,252.31) -- (335.83,210.65) ;
			%Straight Lines [id:da9796000698475377] 
			\draw    (350.33,240.41) -- (343.08,252.31) ;
			
			%Straight Lines [id:da008070566499312593] 
			\draw    (332.33,240.41) -- (339.58,251.72) ;
			%Straight Lines [id:da06372706479601775] 
			\draw    (339.58,251.72) -- (346.83,210.65) ;
			%Straight Lines [id:da362633094740475] 
			\draw    (354.08,252.31) -- (346.83,210.65) ;
			%Straight Lines [id:da6488429274362556] 
			\draw    (361.33,240.41) -- (354.08,252.31) ;
			
			%Straight Lines [id:da8689714606229884] 
			\draw    (342.33,240.41) -- (349.58,251.72) ;
			%Straight Lines [id:da5897763825949154] 
			\draw    (349.58,251.72) -- (356.83,210.65) ;
			%Straight Lines [id:da6417688745872986] 
			\draw    (364.08,252.31) -- (356.83,210.65) ;
			%Straight Lines [id:da9185627515782969] 
			\draw    (371.33,240.41) -- (364.08,252.31) ;
			
			%Straight Lines [id:da08271889357806428] 
			\draw    (331.33,195.41) -- (335.54,206.72) ;
			%Straight Lines [id:da3795435558329564] 
			\draw    (335.54,206.72) -- (339.74,165.65) ;
			%Straight Lines [id:da6682655820740195] 
			\draw    (343.95,207.31) -- (339.74,165.65) ;
			%Straight Lines [id:da6475345620282154] 
			\draw    (348.15,195.41) -- (343.95,207.31) ;
			
			%Straight Lines [id:da7585028178324489] 
			\draw    (337.71,195.41) -- (341.92,206.72) ;
			%Straight Lines [id:da18858227018798546] 
			\draw    (341.92,206.72) -- (346.12,165.65) ;
			%Straight Lines [id:da2513405515163567] 
			\draw    (350.33,207.31) -- (346.12,165.65) ;
			%Straight Lines [id:da8226431915518251] 
			\draw    (354.53,195.41) -- (350.33,207.31) ;
			
			%Straight Lines [id:da10304239703471008] 
			\draw    (343.51,195.41) -- (347.72,206.72) ;
			%Straight Lines [id:da929394911907056] 
			\draw    (347.72,206.72) -- (351.92,165.65) ;
			%Straight Lines [id:da5976385084710379] 
			\draw    (356.13,207.31) -- (351.92,165.65) ;
			%Straight Lines [id:da7624001718580387] 
			\draw    (360.33,195.41) -- (356.13,207.31) ;

			%Straight Lines [id:da32070168192284854] 
			\draw    (337.33,152.41) -- (339.8,163.72) ;
			%Straight Lines [id:da45779721339936086] 
			\draw    (339.8,163.72) -- (342.26,122.65) ;
			%Straight Lines [id:da3170719341608639] 
			\draw    (344.73,164.31) -- (342.26,122.65) ;
			%Straight Lines [id:da24257443768340825] 
			\draw    (347.19,152.41) -- (344.73,164.31) ;
			
			%Straight Lines [id:da10435343135213804] 
			\draw    (341.07,152.41) -- (343.54,163.72) ;
			%Straight Lines [id:da42898433892980614] 
			\draw    (343.54,163.72) -- (346,122.65) ;
			%Straight Lines [id:da31076167099727225] 
			\draw    (348.47,164.31) -- (346,122.65) ;
			%Straight Lines [id:da5499015821558104] 
			\draw    (350.93,152.41) -- (348.47,164.31) ;
			
			%Straight Lines [id:da7769480457548112] 
			\draw    (344.47,152.41) -- (346.94,163.72) ;
			%Straight Lines [id:da7192791829281896] 
			\draw    (346.94,163.72) -- (349.4,122.65) ;
			%Straight Lines [id:da5347274565657548] 
			\draw    (351.87,164.31) -- (349.4,122.65) ;
			%Straight Lines [id:da7063062076311363] 
			\draw    (354.33,152.41) -- (351.87,164.31) ;

			%Straight Lines [id:da3975887813728973] 
			\draw    (342.33,66.41) -- (343.11,77.72) ;
			%Straight Lines [id:da034221561904259845] 
			\draw    (343.11,77.72) -- (343.88,36.65) ;
			%Straight Lines [id:da1947286786114495] 
			\draw    (344.65,78.31) -- (343.88,36.65) ;
			%Straight Lines [id:da4364842906459272] 
			\draw    (345.43,66.41) -- (344.65,78.31) ;
			
			%Straight Lines [id:da7549715421343396] 
			\draw    (343.51,66.41) -- (344.28,77.72) ;
			%Straight Lines [id:da6370578351355539] 
			\draw    (344.28,77.72) -- (345.05,36.65) ;
			%Straight Lines [id:da31006930677389166] 
			\draw    (345.83,78.31) -- (345.05,36.65) ;
			%Straight Lines [id:da09182472091606453] 
			\draw    (346.6,66.41) -- (345.83,78.31) ;
			
			%Straight Lines [id:da40530414353147826] 
			\draw    (344.57,66.41) -- (345.35,77.72) ;
			%Straight Lines [id:da15423341577139038] 
			\draw    (345.35,77.72) -- (346.12,36.65) ;
			%Straight Lines [id:da09763430840057508] 
			\draw    (346.89,78.31) -- (346.12,36.65) ;
			%Straight Lines [id:da4576842199722658] 
			\draw    (347.67,66.41) -- (346.89,78.31) ;

			%Straight Lines [id:da7084672280740796] 
			\draw    (110,253.64) -- (110,266.64) ;
			%Straight Lines [id:da9901211633251319] 
			\draw    (510,253.64) -- (510,266.64) ;
			
			% Text Node
			\draw (341,87) node [anchor=north west][inner sep=0.75pt]    {$\vdots$};
			% Text Node
			\draw (390,53) node [anchor=north west][inner sep=0.75pt]  [font=\footnotesize]  {$(2m-2) J-1$};
			% Text Node
			\draw (390,133) node [anchor=north west][inner sep=0.75pt]  [font=\footnotesize]  {$J+3$};
			% Text Node
			\draw (390,182) node [anchor=north west][inner sep=0.75pt]  [font=\footnotesize]  {$J+2$};
			% Text Node
			\draw (390,222) node [anchor=north west][inner sep=0.75pt]  [font=\footnotesize]  {$J+1$};
			% Text Node
			\draw (341,270) node [anchor=north west][inner sep=0.75pt]    {$\Gamma $};
			% Text Node
			\draw (105,270) node [anchor=north west][inner sep=0.75pt]    {$0$};
			% Text Node
			\draw (504,270) node [anchor=north west][inner sep=0.75pt]    {$1$};
		\end{tikzpicture}
		\caption{An illustration of $\cup_{j=J+1}^{(2m-2)J-1} [2^{-j}\mathcal{S}_j]$ in \eqref{BSH1J0J}.}
		\label{Sj}
	\end{figure}
	
	Below we present the main result of this paper on the optimal convergence rates of our wavelet Galerkin method and the uniform boundedness of the condition number of the coefficient matrix. The detailed and self-contained proof of the following main result is presented in \cref{sec:proof}.

	\begin{theorem} \label{thm:main}
		Let $m\ge 2$ be an integer and $\Gamma$ be a finite subset of the domain $\Omega:=(0,1)$. Assume that
$u \in H^{1}_0(\Omega)$ is the true solution of the model problem \eqref{model} with variable functions $a,f$ such that
%$u_+ \in H^{m}(\Omega_+)$ and $u_- \in H^{m}(\Omega_-)$.
$u\in H^m(\Omega\setminus \Gamma)$.
Define $V^{wav}_{h} := \text{span}(\mathcal{B}^{S,H^{1}_0(\Omega)}_{J_0,J})$, where $\mathcal{B}^{S,H^{1}_0(\Omega)}_{J_0,J}$ is defined in \eqref{BSH1J0J} and constructed from a compactly supported biorthogonal wavelet $(\{\tilde{\phi};\tilde{\psi}\},\{\phi;\psi\})$ in $L^{2}(\R)$ such that $\phi,\psi \in H^1(\R)$ and  $\tilde{\psi}$ has $m$ vanishing moments:
\begin{equation}\label{vm:psi}
\vmo(\tilde{\psi})\ge m,\qquad \mbox{i.e.},\quad \int_\R x^j \tilde{\psi}(x)dx=0,\qquad j=0,\ldots, m-1,
\end{equation}
and the basic compactly supported boundary dual wavelets $\tilde{\psi}^L$ on $[0,\infty)$ and $\tilde{\psi}^R$ on $(-\infty,1]$ satisfy
\begin{equation}\label{vm:LR}
\int_0^\infty x^j \tilde{\psi}^L(x) dx=\int_{-\infty}^1 (x-1)^j \tilde{\psi}^R(x) dx=0,\qquad j=1,\ldots,m-1.
\end{equation}
For $J \ge J_0$,
let $h:=2^{-J}$ and $N_J:=\#\mathcal{B}^{S,H^{1}_0(\Omega)}_{J_0,J}$ for its cardinality.
Note that $N_J^{-1}=\bo(h)$.
Define $u_h = u_J := \sum_{\eta \in \mathcal{B}^{S,H^{1}_0(\Omega)}_{J_0,J}} c_\eta \eta \in V^{wav}_h$ as the approximate solution from solving \eqref{BuJv}. Then for all $J \ge J_0$, there is a positive constant $C$, independent of $J$, $h$, $N_J$, such that
		\[
		\|u - u_h \|_{H^{1}(\Omega)} \le C h^{m-1}, \qquad \|u - u_h \|_{L^{2}(\Omega)} \le C h^{m},
		\]
		and
		\[
\|u - u_J \|_{H^{1}(\Omega)} \le C N_J^{1-m},
 \qquad \|u - u_J \|_{L^{2}(\Omega)} \le C N_J^{-m},
		\]
		where $C$ is a generic constant bounded above by $c \|u\|_{H^{m}(\Omega\setminus \Gamma)}$
%$c (\|u\|_{H^{m}(\Omega_-)} + \|u\|_{H^{m}(\Omega_+)})$
with the constant $c$ depending only on $\Gamma$, the function $a$, and the wavelet basis. Moreover, the condition number of the coefficient matrix, $\kappa$, satisfies
		\[ \kappa\Big([\blf(\alpha,\beta)]_{\alpha,\beta\in \mathcal{B}^{S,H^1_0(\Omega)}_{J_0,J}}\Big) \le C_w \| a\|_{L^\infty(\Omega)} \|a^{-1}\|_{L^\infty(\Omega)}, \quad \text{for all } J \ge J_0,
		\]
		where $B(\cdot,\cdot)$ is defined in \eqref{BuJv}, and $C_w$ is a constant that only depends on the wavelet basis.
	\end{theorem}

\iffalse
Because we are using
compactly supported biorthogonal wavelets $(\{\tilde{\phi};\tilde{\psi}\},\{\phi;\psi\})$ in $L^{2}(\R)$, the condition in \eqref{vm:psi} of $m$ vanishing moments for the dual wavelet $\tilde{\psi}$ is equivalent to the polynomial reproduction property
$\{1,x,\ldots,x^{m-1}\}\subset \mbox{span}\{\phi(\cdot-k) \setsp k\in \Z\}$.
The conditions in \eqref{vm:LR} for the compactly supported boundary dual wavelets $\tilde{\psi}^L$ and $\tilde{\psi}^R$ are equivalent to
\[
x^j|_{[0,\infty)} \in \mbox{span}\{\phi^L, \phi(\cdot-k) \setsp k \ge n_{l,\phi}\}
\quad \mbox{and}\quad
(x-1)^j|_{(-\infty,-1]} \in \mbox{span}\{\phi^R, \phi(\cdot-k) \setsp k \le n_{h,\phi}\}
\]
for all $j=1,\ldots, m-1$.
However, due to the homogeneous boundary conditions of the locally supported biorthogonal wavelet $(\tilde{\mathcal{B}}_{J_0},\mathcal{B}_{J_0})$ in $L^2(\Omega)$,
we must conclude that
%
\[
\int_0^\infty \tilde{\psi}^L(x) dx\ne 0 \quad \text{and} \quad
\int_{-\infty}^1 \tilde{\psi}^R(x) dx\ne 0.
\]
%
See \cite{hanbook,HM21} for details.
By the same proof in \cref{sec:proof}, \cref{thm:main} also holds if $\Gamma\subset \Omega$ is a finite set.
\fi

\section{Numerical Experiments} \label{sec:numerical}
	For this paper, we use three kinds of locally supported biorthogonal wavelets on the unit interval $[0,1]$ with increasing approximation orders $m$, as presented in \cite[Examples 3.1, 3.4, and 3.5]{HM25} satisfying the conditions \eqref{vm:psi} and \eqref{vm:LR} with $m=2,3,4$ respectively, and the later two being multiwavelets (since $\phi,\psi,\tilde{\phi},\tilde{\psi}$ are $2 \times 1$ vector-valued functions). In the following tables, we define $\text{E}_{L^2}:= \|u_{J}-u\|_{L^2(\Omega)}$ and  $\text{E}_{H^1}:= \|u'_{J}-u'\|_{L^2(\Omega)}$, where $u_J$ is the approximate solution obtained from \eqref{BuJv}. Otherwise, we may replace $u$ with a reference solution $u_{\text{ref}}$ computed using a fine grid. The errors measured in the $L^2(\Omega)$-norm are approximated by the Riemann sum on a fine grid. Because $N_J:=\#\mathcal{B}^{S,H^1_0(\Omega)}_{J_0,J}$ and $N_J^{-1}=\bo(h)$ with $h=2^{-J}$,
in accordance with \cref{thm:main},
we record a convergence rate that depends on $h$ and another on $N_J$:
	\begin{equation} \label{order}
	\text{Ord}_{L^2,h} := \log_2\left( \frac{\|u_{J-1}-u\|_{L^2(\Omega)}}{\|u_{J}-u\|_{L^2(\Omega)}} \right),
	\quad
	\text{Ord}_{L^2,N} := d \frac{\log_2\left( \|u_{J-1}-u\|_{L^2(\Omega)} / \|u_{J}-u\|_{L^2(\Omega)} \right)}{\log_2\left( N_J / N_{J-1} \right)},
	\end{equation}
	where $d=1$ if the problem is 1D and $d=2$ if the problem is 2D. The convergence rates in the $H^1(\Omega)$-semi-norm, denoted by $\text{Ord}_{H^1,h}$ and $\text{Ord}_{H^1,N}$, are similarly computed
	with the function replaced by its derivative. The condition number of the coefficient matrix from our wavelet Galerkin method is reported in the column labeled $\kappa$. If $\Gamma$ contains only a single point (i.e., there is only one interface point), define $\Omega_- := (0,\Gamma)$ and $\Omega_+ := (\Gamma,1)$. In the following tables, we compare the numerical performance of using \eqref{BH1J0J} (the unenriched wavelet basis) against \eqref{BSH1J0J} (the enriched wavelet basis).  
	
	\begin{example} \label{exnogamma}
		\normalfont
		Consider the model problem \eqref{model} with $a_- = 1$, $a_+ = 10^5$, $\Gamma:=\pi/6$, and $f$ is chosen such that the true solution $u$ satisfies $u_- = x e^x$, and
		\be \label{ex:u}
		\begin{aligned}
		u_+ & = \left(\tfrac{\left(-3 \Gamma^{4}+20 \Gamma^{3} a_+ +\left(-5 a_+ +6\right) \Gamma^{2}-16 a_+ \Gamma +9 a_+ -3\right) {\mathrm e}^{\Gamma}}{\Gamma  \left(\Gamma -1\right)^{3} a_+ \left(\Gamma -3\right)}\right)x^2 \\
		& \; + \left(\tfrac{\left(3 \Gamma^{5}+\left(-20 a_+ +7\right) \Gamma^{4}+\left(-40 a_+ -10\right) \Gamma^{3}+\left(50 a_+ -10\right) \Gamma^{2}+\left(-8 a_+ +7\right) \Gamma -6 a_+ +3\right) {\mathrm e}^{\Gamma}}{\left(\Gamma -1\right)^{3} a_+ \left(\Gamma -3\right) \Gamma^{2}}\right)x^3 \\
		& \; + \left(\tfrac{\left(-7 \Gamma^{4}+45 \Gamma^{3} a_+ +\left(-10 a_+ +14\right) \Gamma^{2}-25 a_+ \Gamma +14 a_+ -7\right) {\mathrm e}^{\Gamma}}{\left(\Gamma -1\right)^{3} a_+ \left(\Gamma -3\right) \Gamma^{2}} \right)x^4 + \left(\tfrac{\left(4 \Gamma^{3}+\left(-24 a_+ -4\right) \Gamma^{2}+\left(24 a_+ -4\right) \Gamma -8 a_+ +4\right) {\mathrm e}^{\Gamma}}{\left(\Gamma -1\right)^{3} a_+ \left(\Gamma -3\right) \Gamma^{2}}\right)x^5.
		\end{aligned}
		\ee
		For this problem, $g_\Gamma =0$. See \cref{ex:nogamma} for the numerical experiments.
		\begin{table}[htbp]
			\begin{center}
				 \resizebox{\textwidth}{!}{%
					 \begin{tabular}{c | c c c c c c c c | c c c c c c c c | c c c c c c c c}
						\hline
						& \multicolumn{8}{c}{FEM version $\mathcal{B}^{H^1_0(\Omega)}_{2,J}$ using \cite[Example 3.1]{HM25} with approximation order $2$} \vline & \multicolumn{8}{c}{FEM version $\mathcal{B}^{H^1_0(\Omega)}_{2,J}$ using \cite[Example 3.4]{HM25} with approximation order $3$} \vline &
						 \multicolumn{8}{c}{FEM version $\mathcal{B}^{H^1_0(\Omega)}_{3,J}$ using \cite[Example 3.5]{HM25} with approximation order $4$}  \\
						\hline
						$J$ & $N_{J}$ & $\kappa$ & $\text{E}_{L^2}$ & \multicolumn{2}{c}{$\text{Ord}_{L^2,h}=\text{Ord}_{L^2,N}$} & $\text{E}_{H^1}$ & \multicolumn{2}{c}{$\text{Ord}_{H^1,h}=\text{Ord}_{H^1,N}$} \vline & $N_{J}$ & $\kappa$ & $\text{E}_{L^2}$ & \multicolumn{2}{c}{$\text{Ord}_{L^2,h}=\text{Ord}_{L^2,N}$} & $\text{E}_{H^1}$ & \multicolumn{2}{c}{$\text{Ord}_{H^1,h}=\text{Ord}_{H^1,N}$} \vline & $N_{J}$ & $\kappa$ & $\text{E}_{L^2}$ & \multicolumn{2}{c}{$\text{Ord}_{L^2,h}=\text{Ord}_{L^2,N}$} & $\text{E}_{H^1}$ & \multicolumn{2}{c}{$\text{Ord}_{H^1,h}=\text{Ord}_{H^1,N}$} \\
						\hline
						2 & 7 & 1.68E+4 & 4.67E-2 &  \multicolumn{2}{c}{}  & 1.05  & \multicolumn{2}{c}{} \vline & 15 & 1.64E+5 & 2.55E-2 & \multicolumn{2}{c}{} & 4.25E-1 & \multicolumn{2}{c}{} \vline & & & & &  & & & \\
						3 & 15 & 2.39E+4 & 2.74E-2  & \multicolumn{2}{c}{0.70} & 6.41E-1 & \multicolumn{2}{c}{0.65} \vline& 31  & 2.06E+5 & 2.50E-2  & \multicolumn{2}{c}{0.03} & 3.99E-1 & \multicolumn{2}{c}{0.09} \vline & 32 & 1.28E+5 & 3.25E-2 & & & 4.52E-1 & & \\
						4 & 31 & 2.86E+4 & 2.52E-2 & \multicolumn{2}{c}{0.12}  & 4.71E-1 & \multicolumn{2}{c}{0.42} \vline& 63 & 2.35E+5 & 2.48E-2  & \multicolumn{2}{c}{0.01} & 3.95E-1 & \multicolumn{2}{c}{0.02} \vline & 64 & 2.05E+5 & 2.52E-2 & \multicolumn{2}{c}{0.37} & 3.76E-1 & \multicolumn{2}{c}{0.27} \\
						5 & 63 & 1.95E+5 & 8.59E-3 & \multicolumn{2}{c}{1.52} & 2.63E-1 & \multicolumn{2}{c}{0.82} \vline& 127 & 2.59E+5 & 8.50E-3  & \multicolumn{2}{c}{1.52} & 2.30E-1 & \multicolumn{2}{c}{0.77} \vline  & 128 & 3.70E+5 & 1.04E-2& \multicolumn{2}{c}{1.27} & 2.54E-1 & \multicolumn{2}{c}{0.56}\\
						6 & 127 &3.49E+5 & 2.42E-4 & \multicolumn{2}{c}{5.09} & 7.15E-2  & \multicolumn{2}{c}{1.86}\vline & 255 & 4.86E+5  & 1.73E-4  & \multicolumn{2}{c}{5.59} & 3.27E-2 & \multicolumn{2}{c}{2.80} \vline & 256 & 4.27E+5 & 1.17E-3 & \multicolumn{2}{c}{3.16} &8.47E-2  & \multicolumn{2}{c}{1.58} \\
						7 & 255 & 4.22E+5 & 1.80E-4 & \multicolumn{2}{c}{0.42} & 4.56E-2 & \multicolumn{2}{c}{0.65}\vline & 511  & 5.49E+5 & 1.73E-4 & \multicolumn{2}{c}{-4.20E-4} & 3.27E-2 & \multicolumn{2}{c}{-3.64E-4} \vline & 512 & 5.06E+5 & 6.69E-4 & \multicolumn{2}{c}{0.80} & 6.41E-2 & \multicolumn{2}{c}{0.40} \\
						8 & 511 & 4.64E+5 & 1.74E-4 & \multicolumn{2}{c}{0.05} & 3.63E-2 & \multicolumn{2}{c}{0.32}\vline & 1023 & 5.61E+5 & 1.73E-4  & \multicolumn{2}{c}{-1.04E-4} & 3.27E-2 & \multicolumn{2}{c}{-1.01E-4} \vline & 1024 & 5.84E+5 & 4.21E-4 & \multicolumn{2}{c}{0.67} & 5.09E-2 & \multicolumn{2}{c}{0.33} \\
						9 & 1023 & 4.89E+5 & 1.73E-4 & \multicolumn{2}{c}{0.01} & 3.36E-2 &\multicolumn{2}{c}{0.11}\vline & 2047 & 5.68E+5 & 1.73E-4  & \multicolumn{2}{c}{-1.88E-5} & 3.27E-2 & \multicolumn{2}{c}{-1.82E-5} \vline & 2048 & 6.42E+5 & 2.97E-4 & \multicolumn{2}{c}{0.50} & 4.28E-2 & \multicolumn{2}{c}{0.25} \\
						10 & 2047 & 5.04E+5 & 1.73E-4 & \multicolumn{2}{c}{0.001} & 3.29E-2 & \multicolumn{2}{c}{0.03}\vline & 4095 & 5.78E+5 & 1.73E-4  & \multicolumn{2}{c}{3.48E-5} & 3.27E-2 & \multicolumn{2}{c}{3.23E-5} \vline & 4096  & 6.74E+5 & 2.35E-5 & \multicolumn{2}{c}{0.34} & 3.80E-2 & \multicolumn{2}{c}{0.17} \\
						\hline
						& \multicolumn{8}{c}{$\mathcal{B}^{S,H^1_0(\Omega)}_{2,J}$ using \cite[Example 3.1]{HM25} with approximation order $2$} \vline & \multicolumn{8}{c}{$\mathcal{B}^{S,H^1_0(\Omega)}_{2,J}$ using \cite[Example 3.4]{HM25} with approximation order $3$} \vline &
						 \multicolumn{8}{c}{$\mathcal{B}^{S,H^1_0(\Omega)}_{3,J}$ using \cite[Example 3.5]{HM25} with approximation order $4$}  \\
						\hline
						$J$ & $N_{J}$ & $\kappa$ & $\|u_{J}-u\|_2$ & $\text{Ord}_{L^2,h}$ & $\text{Ord}_{L^2,N}$ & $\| u'_J -  u'\|_2 $ & $\text{Ord}_{H^1,h}$ & $\text{Ord}_{H^1,N}$ & $N_{J}$ & $\kappa$ & $\|u_{J}-u\|_2$ & $\text{Ord}_{L^2,h}$ & $\text{Ord}_{L^2,N}$ & $\|u'_J - u'\|_2 $ & $\text{Ord}_{H^1,h}$ & $\text{Ord}_{H^1,N}$ & $N_{J}$ & $\kappa$ & $\|u_{J}-u\|_2$ & $\text{Ord}_{L^2,h}$ & $\text{Ord}_{L^2,N}$ & $\|u_J' - u'\|_2 $ & $\text{Ord}_{H^1,h}$ & $\text{Ord}_{H^1,N}$ \\
						\hline
						2 & 10 & 2.08E+4 &4.64E-2 &  & & 1.05 & & & 55 & 5.46E+5  & 2.41E-3 &  & & 1.31E-1 & & & & & & & & & & \\
						3 & 21  & 1.94E+5 &1.34E-2 & 1.79 & 1.68 & 5.53E-1 & 0.93 & 0.86 & 95 &  5.74E+5 & 3.63E-4 & 2.73 & 3.47 & 4.65E-2 & 1.49 & 1.89 & 256 & 8.78E+5 & 1.33E-5 & & & 3.47E-3 & & \\
						4 & 40 & 4.15E+5 & 2.52E-3 & 2.41 & 2.59  & 2.56E-1 & 1.11 & 1.20 & 151 & 5.99E+5  & 4.11E-5 & 3.14 & 4.70 & 1.12E-2 & 2.05 & 3.06 & 368 & 8.78E+5 & 1.01E-6 & 3.71 & 7.08 & 5.61E-4 & 2.63 & 5.02 \\
						5 & 75 & 4.80E+5 & 6.62E-4 & 1.93 & 2.13  & 1.31E-1 & 0.96 & 1.06 & 239 & 7.23E+5  &  5.01E-6 & 3.04 & 4.58 & 2.24E-3 & 2.32 & 3.51 & 512 & 1.02E+6 & 6.64E-8 & 3.93 & 8.25 & 6.04E-5 & 3.22 & 6.75 \\
						6 & 142 & 5.07E+5 & 2.42E-4 & 1.45 & 1.58 & 7.15E-2 & 0.87 & 0.95 & 391 & 7.38E+5  & 6.29E-7 & 3.00 & 4.22 & 7.53E-4 & 1.57 & 2.21 & 720 & 1.06E+6 & 4.23E-9 & 3.97 & 8.07 & 9.07E-6 & 2.74 & 5.56 \\
						7 & 273 & 5.26E+5 & 5.95E-5 & 2.02 & 2.14 & 3.57E-2 & 1.00 & 1.06  & & &  &  &  &  & & &  &  &  & & & & & \\
						8 & 532 & 5.66E+5 & 1.40E-5 & 2.09 & 2.17 & 1.76E-2 & 1.02 &  1.06 &  &   &  &  & & & & & & & & & & & & \\
						9 & 1047 & 6.22E+5 & 2.78E-6 & 2.33 & 2.39 & 8.38E-3& 1.07 & 1.10 &  &   &  &  & & & & & & & & & & & & \\
						10 & 2074 & 6.53E+5 & 6.32E-7 & 2.14 & 2.17 & 4.06E-3 & 1.04 & 1.06  &  &   &  &  & & & & & & & & \\
						\hline
					\end{tabular}%
				}
			\end{center}
			\caption{Numerical results for \cref{exnogamma}.} %The slopes are computed by performing a linear regression on $\log_2(\|u_J-u\|_2)$ and $\log_2(\|u'_J-u'\|_2)$ versus $\log_2(N_J)$.}
			\label{ex:nogamma}
		\end{table}
	\end{example}
	
	\begin{example}
		\label{ex:threebp}
		\normalfont
		Consider the model problem \eqref{model} with $a(x) = 1 \chi_{[0,2/5)} + 10^3 \chi_{[2/5,\pi/6)} + 10^2 \chi_{[\pi/6,3/5)} + 10^5 \chi_{[3/5,1]}$, $\Gamma:=\{2/5,\pi/6,3/5\}$, and $f$ is chosen such that the true solution satisfies $u(x)= x e^x$ for $x \in [0,2/5)$, $u(x)$ takes the form of \eqref{ex:u} for $x \in [3/5,1]$, and
		\begin{align*}
			u(x) & = \tfrac{9 x^{4} \left(5 \pi^{2} {\mathrm e}^{\frac{\pi}{6}}+12500 \pi^{2} {\mathrm e}^{\frac{2}{5}}-89982 \pi  \,{\mathrm e}^{\frac{\pi}{6}}+143928 \,{\mathrm e}^{\frac{\pi}{6}}\right)}{50 \left(5 \pi -12\right)^{2} \pi^{2}}-\tfrac{3 x^{3} \left(25 \pi^{3} {\mathrm e}^{\frac{\pi}{6}}+125000 \pi^{3} {\mathrm e}^{\frac{2}{5}}-599850 \pi^{2} {\mathrm e}^{\frac{\pi}{6}}-144 \pi  \,{\mathrm e}^{\frac{\pi}{6}}+1727136 \,{\mathrm e}^{\frac{\pi}{6}}\right)}{500 \left(5 \pi -12\right)^{2} \pi^{2}}\\
			& +\tfrac{x^{2} \left(15625 \pi^{3} {\mathrm e}^{\frac{2}{5}}+15 \pi^{2} {\mathrm e}^{\frac{\pi}{6}}-359946 \pi  \,{\mathrm e}^{\frac{\pi}{6}}+647784 \,{\mathrm e}^{\frac{\pi}{6}}\right)}{250 \pi  \left(5 \pi -12\right)^{2}}, \quad x \in [2/5,\pi/6),\\
			u(x) & = \tfrac{9 x^{4} {\mathrm e}^{\frac{\pi}{6}} \left(6250 \pi^{5}-11437545 \pi^{4}+572850216 \pi^{3}-5499898056 \pi^{2}+20322892224 \pi -26681411664\right)}{31250 \pi^{2} \left(5 \pi^{5}-198 \pi^{4}+2808 \pi^{3}-18576 \pi^{2}+58320 \pi -69984\right)}\\
			& -\tfrac{3 x^{3} \left(3125 \pi^{6}-7582545 \pi^{5}+423337716 \pi^{4}-3392818056 \pi^{3}+6424102224 \pi^{2}+16883628336 \pi -52225560000\right) {\mathrm e}^{\frac{\pi}{6}}}{31250 \left(5 \pi^{5}-198 \pi^{4}+2808 \pi^{3}-18576 \pi^{2}+58320 \pi -69984\right) \pi^{2}}\\
			& +\tfrac{9 x^{2} {\mathrm e}^{\frac{\pi}{6}} \left(2495 \pi^{5}+13925024 \pi^{4}+91140216 \pi^{3}-2378520864 \pi^{2}+11576518704 \pi -17437680000\right)}{125000 \pi  \left(5 \pi^{5}-198 \pi^{4}+2808 \pi^{3}-18576 \pi^{2}+58320 \pi -69984\right)}, \quad x \in [\pi/6,3/5).
		\end{align*}
		For this problem, $g_{2/5}, g_{3/5} \neq 0$, and $g_{\pi/6} = 0$. See \cref{tab:threebp} for the numerical experiments and \cref{err:threebp} for the error plots of $u_J-u$ and $u'_J-u'$. 
		\begin{table}[htbp]
			\begin{center}
				\resizebox{\textwidth}{!}{%
					\begin{tabular}{c | c c c c c c c c | c c c c c c c c | c c c c c c c c}
						\hline
						& \multicolumn{8}{c}{FEM version $\mathcal{B}^{H^1_0(\Omega)}_{2,J}$ using \cite[Example 3.1]{HM25} with approximation order $2$} \vline & \multicolumn{8}{c}{FEM version $\mathcal{B}^{H^1_0(\Omega)}_{2,J}$ using \cite[Example 3.4]{HM25} with approximation order $3$} \vline &
						\multicolumn{8}{c}{FEM version $\mathcal{B}^{H^1_0(\Omega)}_{3,J}$ using \cite[Example 3.5]{HM25} with approximation order $4$}  \\
						\hline
						$J$ & $N_J$ & $\kappa$ & $\mathrm{E}_{L^2}$ & \multicolumn{2}{c}{$\mathrm{Ord}_{L^2,h}=\mathrm{Ord}_{L^2,N}$} & $\mathrm{E}_{H^1}$ & \multicolumn{2}{c|}{$\mathrm{Ord}_{H^1,h}=\mathrm{Ord}_{H^1,N}$} & $N_J$ & $\kappa$ & $\mathrm{E}_{L^2}$ & \multicolumn{2}{c}{$\mathrm{Ord}_{L^2,h}=\mathrm{Ord}_{L^2,N}$} & $\mathrm{E}_{H^1}$ & \multicolumn{2}{c|}{$\mathrm{Ord}_{H^1,h}=\mathrm{Ord}_{H^1,N}$} & $N_J$ & $\kappa$ & $\mathrm{E}_{L^2}$ & \multicolumn{2}{c}{$\mathrm{Ord}_{L^2,h}=\mathrm{Ord}_{L^2,N}$} & $\mathrm{E}_{H^1}$ & \multicolumn{2}{c}{$\mathrm{Ord}_{H^1,h}=\mathrm{Ord}_{H^1,N}$} \\
						\hline
						2 & 7 & 1.16E+04 & 3.87E-02 & \multicolumn{2}{c}{} & 1.01E+00 & \multicolumn{2}{c|}{} & 15 & 2.35E+04 & 2.40E-02 & \multicolumn{2}{c}{} & 4.00E-01 & \multicolumn{2}{c|}{} & & & & & & & & \\
						3 & 15 & 1.58E+04 & 2.47E-02 & \multicolumn{2}{c}{0.59} & 5.87E-01 & \multicolumn{2}{c|}{0.71} & 31 & 3.59E+04 & 1.45E-02 & \multicolumn{2}{c}{0.70} & 3.00E-01 & \multicolumn{2}{c|}{0.40} & 32 & 1.61E+03 & 1.73E-02 & \multicolumn{2}{c}{} & 2.87E-01 & \multicolumn{2}{c}{} \\
						4 & 31 & 1.84E+04 & 1.60E-02 & \multicolumn{2}{c}{0.60} & 3.77E-01 & \multicolumn{2}{c|}{0.61} & 63 & 5.11E+04 & 1.13E-02 & \multicolumn{2}{c}{0.35} & 2.13E-01 & \multicolumn{2}{c|}{0.48} & 64 & 3.24E+03 & 5.64E-03 & \multicolumn{2}{c}{1.61} & 1.62E-01 & \multicolumn{2}{c}{0.83} \\
						5 & 63 & 2.06E+04 & 5.71E-03 & \multicolumn{2}{c}{1.45} & 2.04E-01 & \multicolumn{2}{c|}{0.86} & 127 & 7.37E+04 & 5.01E-03 & \multicolumn{2}{c}{1.16} & 1.53E-01 & \multicolumn{2}{c|}{0.47} & 128 & 3.94E+03 & 3.64E-03 & \multicolumn{2}{c}{0.63} & 1.08E-01 & \multicolumn{2}{c}{0.59} \\
						6 & 127 & 2.26E+04 & 7.89E-04 & \multicolumn{2}{c}{2.82} & 9.17E-02 & \multicolumn{2}{c|}{1.14} & 255 & 9.38E+04 & 8.02E-04 & \multicolumn{2}{c}{2.63} & 6.20E-02 & \multicolumn{2}{c|}{1.30} & 256 & 4.56E+03 & 1.35E-03 & \multicolumn{2}{c}{1.43} & 7.75E-02 & \multicolumn{2}{c}{0.47} \\
						7 & 255 & 2.42E+04 & 8.31E-04 & \multicolumn{2}{c}{-0.07} & 7.03E-02 & \multicolumn{2}{c|}{0.38} & 511 & 9.64E+04 & 8.34E-04 & \multicolumn{2}{c}{-0.06} & 6.16E-02 & \multicolumn{2}{c|}{0.01} & 512 & 5.07E+03 & 1.03E-03 & \multicolumn{2}{c}{0.39} & 6.52E-02 & \multicolumn{2}{c}{0.25} \\
						8 & 511 & 2.61E+04 & 8.68E-04 & \multicolumn{2}{c}{-0.06} & 6.38E-02 & \multicolumn{2}{c|}{0.14} & 1023 & 1.00E+05 & 6.80E-04 & \multicolumn{2}{c}{0.29} & 4.85E-02 & \multicolumn{2}{c|}{0.34} & 1024 & 5.88E+03 & 3.54E-04 & \multicolumn{2}{c}{1.54} & 3.89E-02 & \multicolumn{2}{c}{0.75} \\
						9 & 1023 & 2.65E+04 & 3.19E-04 & \multicolumn{2}{c}{1.44} & 3.86E-02 & \multicolumn{2}{c|}{0.72} & 2047 & 1.05E+05 & 3.13E-04 & \multicolumn{2}{c}{1.12} & 3.67E-02 & \multicolumn{2}{c|}{0.40} & 2048 & 5.94E+03 & 2.22E-04 & \multicolumn{2}{c}{0.67} & 2.63E-02 & \multicolumn{2}{c}{0.56} \\
						10 & 2047 & 2.77E+04 & 4.64E-05 & \multicolumn{2}{c}{2.78} & 1.60E-02 & \multicolumn{2}{c|}{1.27} & 4095 & 1.07E+05 & 4.85E-05 & \multicolumn{2}{c}{2.69} & 1.54E-02 & \multicolumn{2}{c|}{1.25} & 4096 & 5.98E+03 & 8.27E-05 & \multicolumn{2}{c}{1.42} & 1.92E-02 & \multicolumn{2}{c}{0.46} \\
						\hline
						& \multicolumn{8}{c}{$\mathcal{B}^{S,H^1_0(\Omega)}_{2,J}$ using \cite[Example 3.1]{HM25} with approximation order $2$} \vline & \multicolumn{8}{c}{$\mathcal{B}^{S,H^1_0(\Omega)}_{2,J}$ using \cite[Example 3.4]{HM25} with approximation order $3$} \vline &
						\multicolumn{8}{c}{$\mathcal{B}^{S,H^1_0(\Omega)}_{3,J}$ using \cite[Example 3.5]{HM25} with approximation order $4$}  \\
						\hline
						$J$ & $N_J$ & $\kappa$ & $\mathrm{E}_{L^2}$ & $\mathrm{Ord}_{L^2,h}$ & $\mathrm{Ord}_{L^2,N}$ & $\mathrm{E}_{H^1}$ & $\mathrm{Ord}_{H^1,h}$ & $\mathrm{Ord}_{H^1,N}$ & $N_J$ & $\kappa$ & $\mathrm{E}_{L^2}$ & $\mathrm{Ord}_{L^2,h}$ & $\mathrm{Ord}_{L^2,N}$ & $\mathrm{E}_{H^1}$ & $\mathrm{Ord}_{H^1,h}$ & $\mathrm{Ord}_{H^1,N}$ & $N_J$ & $\kappa$ & $\mathrm{E}_{L^2}$ & $\mathrm{Ord}_{L^2,h}$ & $\mathrm{Ord}_{L^2,N}$ & $\mathrm{E}_{H^1}$ & $\mathrm{Ord}_{H^1,h}$ & $\mathrm{Ord}_{H^1,N}$ \\
						\hline
						2 & 11 & 1.29E+04 & 4.41E-02 &  &  & 1.02E+00 &  &  & 109 & 8.86E+04 & 2.55E-03 &  &  & 1.41E-01 &  &  & & & & & & & & \\
						3 & 30 & 1.94E+04 & 1.13E-02 & 1.97 & 1.36 & 5.28E-01 & 0.95 & 0.65 & 211 & 1.07E+05 & 3.22E-04 & 2.99 & 3.13 & 3.64E-02 & 1.95 & 2.05 & 654 & 6.16E+03 & 1.13E-05 &  &  & 2.08E-03 &  &  \\
						4 & 58 & 2.38E+04 & 2.64E-03 & 2.10 & 2.21 & 2.66E-01 & 0.99 & 1.04 & 325 & 1.14E+05 & 4.02E-05 & 3.00 & 4.82 & 9.19E-03 & 1.98 & 3.19 & 952 & 6.18E+03 & 9.54E-07 & 3.56 & 6.58 & 3.22E-04 & 2.69 & 4.97 \\
						5 & 99 & 2.63E+04 & 7.14E-04 & 1.88 & 2.44 & 1.38E-01 & 0.94 & 1.22 & 463 & 1.17E+05 & 5.50E-06 & 2.87 & 5.62 & 2.49E-03 & 1.88 & 3.69 & 1274 & 6.19E+03 & 6.49E-08 & 3.88 & 9.22 & 3.70E-05 & 3.12 & 7.42 \\
						6 & 172 & 2.86E+04 & 1.72E-04 & 2.05 & 2.58 & 6.96E-02 & 0.99 & 1.25 & 663 & 1.19E+05 & 7.16E-07 & 2.94 & 5.68 & 6.42E-04 & 1.95 & 3.77 & 1648 & 6.20E+03 & 4.39E-09 & 3.89 & 10.46 & 5.33E-06 & 2.79 & 7.52 \\
						7 & 309 & 2.99E+04 & 4.56E-05 & 1.91 & 2.27 & 3.54E-02 & 0.97 & 1.15 & & & & & & & & & & & & & & & & \\
						8 & 574 & 3.11E+04 & 1.08E-05 & 2.07 & 2.32 & 1.75E-02 & 1.01 & 1.13 & & & & & & & & & & & & & & & & \\
						9 & 1095 & 3.18E+04 & 2.85E-06 & 1.92 & 2.06 & 8.88E-03 & 0.98 & 1.05 & & & & & & & & & & & & & & & & \\
						10 & 2128 & 3.26E+04 & 6.77E-07 & 2.07 & 2.16 & 4.39E-03 & 1.02 & 1.06 & & & & & & & & & & & & & & & & \\
						\hline 
					\end{tabular}%
				}
			\end{center}
			\caption{Numerical results for \cref{ex:threebp}.}
			\label{tab:threebp}
		\end{table}
		\begin{figure}
			\includegraphics[width=0.7\textwidth]{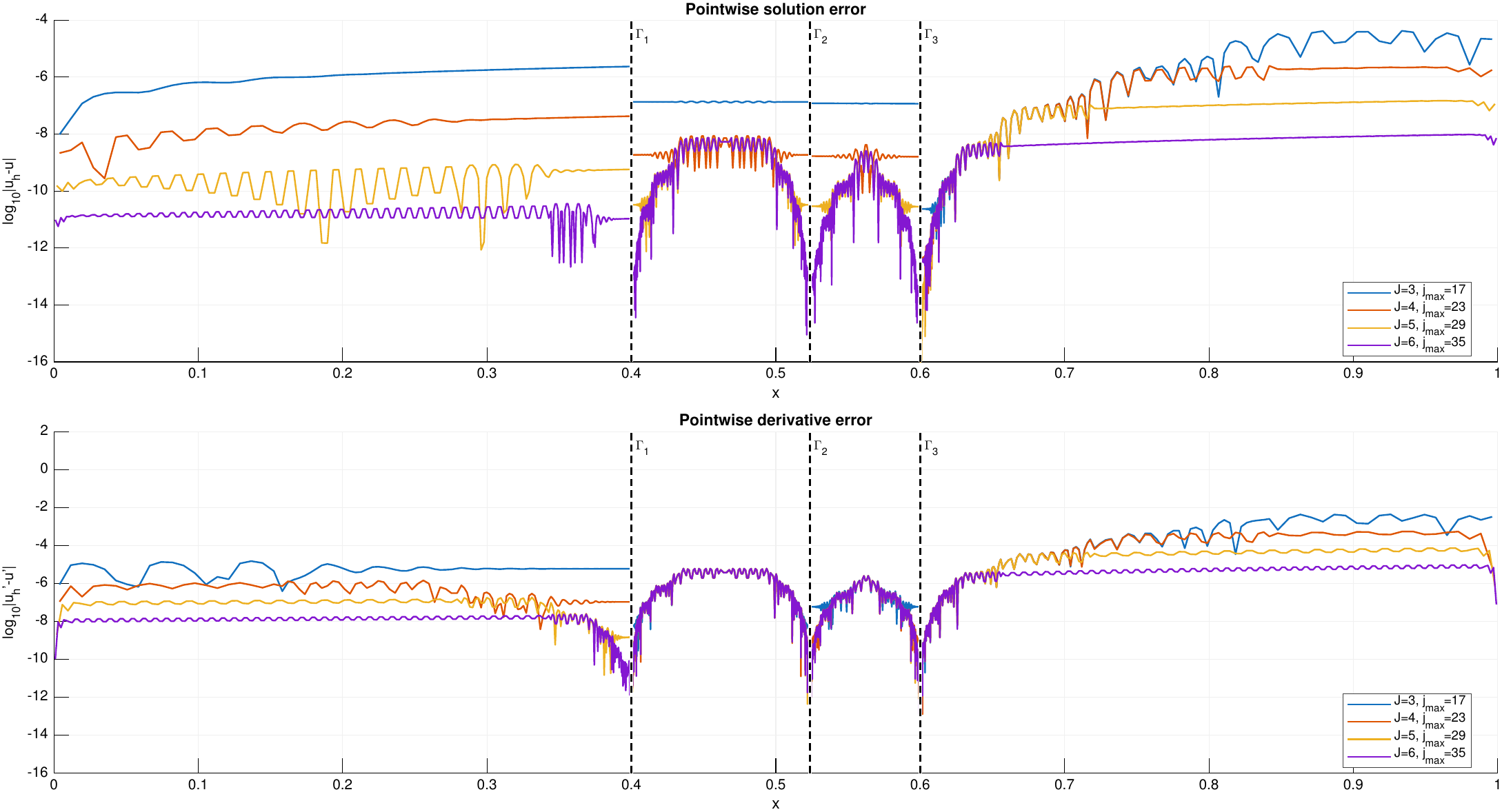}
			\caption{Error plots of $u_J-u$ (top) and $u'_J-u'$ (bottom) in \cref{ex:threebp}.}
			\label{err:threebp}
		\end{figure}
	\end{example}
	
	% Insertion-ready example for WaveEIP1DAug09.tex.
	% The Google Drive manuscript is not modified.
	\begin{example}
		\label{ex:close-threebp}
		\normalfont
		Consider the model problem \eqref{model} with $a(x) = 1 \chi_{[0,2/5)} + 10^3 \chi_{[2/5,\pi/6)} + 10^2 \chi_{[\pi/6,\pi/6+1/50)} + 10^5 \chi_{[\pi/6+1/50,1]}$, $\Gamma:=\{2/5,\pi/6,\pi/6+1/50\}$, $f=-1$, and 
		\[
		g_{2/5}=\tfrac{1820784942559719}{500000000000}, \quad 
		g_{\pi/6}=0, \quad 
		g_{\pi/6+1/50} = \tfrac{497636265803627}{62500000000}.
		\]
		The true solution $u$ can be obtained by integration. Compared to \cref{ex:threebp}, the second and third interface points in $\Gamma$ are much closer to each other. See \cref{tab:close-threebp} for the numerical experiments. Even in the presence of interface points that are very close to each other, our method still performs well. 
		\begin{table}[htbp]
			\begin{center}
				\resizebox{\textwidth}{!}{%
					\begin{tabular}{c | c c c c c c c c | c c c c c c c c | c c c c c c c c}
						\hline
						& \multicolumn{8}{c|}{$\mathcal{B}^{S,H^1_0(\Omega)}_{2,J}$ using \cite[Example 3.1]{HM25} with approximation order $2$}
						& \multicolumn{8}{c|}{$\mathcal{B}^{S,H^1_0(\Omega)}_{2,J}$ using \cite[Example 3.4]{HM25} with approximation order $3$}
						& \multicolumn{8}{c}{$\mathcal{B}^{S,H^1_0(\Omega)}_{3,J}$ using \cite[Example 3.5]{HM25} with approximation order $4$} \\
						\hline
						$J$ & $N_J$ & $\kappa$ & $\mathrm{E}_{L^2}$ & $\mathrm{Ord}_{L^2,h}$ & $\mathrm{Ord}_{L^2,N}$ & $\mathrm{E}_{H^1}$ & $\mathrm{Ord}_{H^1,h}$ & $\mathrm{Ord}_{H^1,N}$
						& $N_J$ & $\kappa$ & $\mathrm{E}_{L^2}$ & $\mathrm{Ord}_{L^2,h}$ & $\mathrm{Ord}_{L^2,N}$ & $\mathrm{E}_{H^1}$ & $\mathrm{Ord}_{H^1,h}$ & $\mathrm{Ord}_{H^1,N}$
						& $N_J$ & $\kappa$ & $\mathrm{E}_{L^2}$ & $\mathrm{Ord}_{L^2,h}$ & $\mathrm{Ord}_{L^2,N}$ & $\mathrm{E}_{H^1}$ & $\mathrm{Ord}_{H^1,h}$ & $\mathrm{Ord}_{H^1,N}$ \\
						\hline
						2  & 11   & 1.56E+04 & 3.82E-01 &      &      & 5.35E+00 &      &      & 93  & 1.21E+05 & 1.79E-02 &      &       & 1.15E+00 &      &      &      &          &          &      &       &          &      &      \\
						3  & 27   & 2.20E+04 & 3.24E-01 & 0.24 & 0.19 & 4.93E+00 & 0.12 & 0.09 & 195 & 1.46E+05 & 3.34E-03 & 2.42 & 2.27  & 4.79E-01 & 1.26 & 1.18 & 622  & 2.01E+04 & 2.67E-05 &      &       & 4.32E-02 &      &      \\
						4  & 53   & 3.01E+04 & 2.62E-02 & 3.63 & 3.73 & 1.16E+00 & 2.09 & 2.14 & 311 & 1.66E+05 & 1.29E-04 & 4.69 & 6.97  & 9.26E-02 & 2.37 & 3.52 & 922  & 2.06E+04 & 3.01E-07 & 6.47 & 11.39 & 4.59E-03 & 3.23 & 5.69 \\
						5  & 96   & 3.33E+04 & 1.90E-02 & 0.47 & 0.55 & 1.12E+00 & 0.05 & 0.06 & 453 & 1.76E+05 & 1.26E-05 & 3.35 & 6.18  & 2.98E-02 & 1.63 & 3.01 & 1248 & 2.08E+04 & 7.87E-09 & 5.26 & 12.04 & 7.46E-04 & 2.62 & 6.00 \\
						6  & 171  & 3.64E+04 & 3.69E-03 & 2.36 & 2.84 & 5.21E-01 & 1.11 & 1.33 & 659 & 1.84E+05 & 2.28E-07 & 5.79 & 10.71 & 4.24E-03 & 2.82 & 5.21 & 1630 & 2.28E+04 & 1.55E-10 & 5.67 & 14.71 & 9.91E-05 & 2.91 & 7.56 \\
						7  & 309  & 3.80E+04 & 6.19E-04 & 2.57 & 3.02 & 2.20E-01 & 1.24 & 1.45 &     &          &          &      &       &          &      &      &      &          &          &      &       &          &      &      \\
						8  & 574  & 3.96E+04 & 1.64E-04 & 1.92 & 2.15 & 1.10E-01 & 1.00 & 1.11 &     &          &          &      &       &          &      &      &      &          &          &      &       &          &      &      \\
						9  & 1095 & 4.05E+04 & 4.01E-05 & 2.03 & 2.18 & 5.29E-02 & 1.06 & 1.14 &     &          &          &      &       &          &      &      &      &          &          &      &       &          &      &      \\
						10 & 2128 & 4.15E+04 & 1.76E-05 & 1.18 & 1.24 & 3.32E-02 & 0.67 & 0.70 &     &          &          &      &       &          &      &      &      &          &          &      &       &          &      &      \\
						\hline
					\end{tabular}%
				}
			\end{center}
			\caption{Numerical results for \cref{ex:close-threebp}.}
			\label{tab:close-threebp}
		\end{table}
	\end{example}

	\begin{example} \label{exgamma}
		\normalfont
		Consider the model problem \eqref{model} with $a_- = 1$, $a_+ = 2 \times 10^4$, $\Gamma:=\sqrt{2}/2$, and $f$ is chosen such that the true solution $u$ satisfies $u_- = e^{x} -(\sin(1-\Gamma) + e^\Gamma -1)x -1$, and
		\[
		\begin{aligned}
			u_+ & = -\sin(-x+\Gamma) + e^\Gamma - (\sin(1-\Gamma) + e^\Gamma -1 )x-1.
		\end{aligned}
		\]
		For this problem, $g_\Gamma = (1-a_+) \sin(1-\Gamma) - a_+ e^\Gamma + 2 a_+ -1$. A similar problem was considered in \cite{HLL11} but with $\Gamma = \pi/6$ and nonzero endpoints. See \cref{ex:gamma} for the numerical experiments.
		\begin{table}[htbp]
			\begin{center}
				 \resizebox{\textwidth}{!}{%
					 \begin{tabular}{c | c c c c c c c c | c c c c c c c c | c c c c c c c c }
						\hline
						& \multicolumn{8}{c}{$\mathcal{B}^{S,H^1_0(\Omega)}_{2,J}$ using \cite[Example 3.1]{HM25} with approximation order $2$} \vline & \multicolumn{8}{c}{$\mathcal{B}^{S,H^1_0(\Omega)}_{2,J}$ using \cite[Example 3.4]{HM25}  with approximation order $3$ }\vline &				 \multicolumn{8}{c}{$\mathcal{B}^{S,H^1_0(\Omega)}_{3,J}$ using \cite[Example 3.5]{HM25} with approximation order $4$}
\\
						\hline
						$J$ & $N_{J}$ & $\kappa$ & $\|u_{J}-u\|_2$ & $\text{Ord}_{L^2,h}$ & $\text{Ord}_{L^2,N}$ & $\| u'_J -  u'\|_2 $ & $\text{Ord}_{H^1,h}$ & $\text{Ord}_{H^1,N}$ & $N_{J}$ & $\kappa$ & $\|u_{J}-u\|_2$ & $\text{Ord}_{L^2,h}$ & $\text{Ord}_{L^2,N}$ & $\|u'_J - u'\|_2 $ & $\text{Ord}_{H^1,h}$ & $\text{Ord}_{H^1,N}$ & $N_{J}$ & $\kappa$ & $\|u_{J}-u\|_2$ & $\text{Ord}_{L^2,h}$ & $\text{Ord}_{L^2,N}$ & $\|u_J' - u'\|_2 $ & $\text{Ord}_{H^1,h}$ & $\text{Ord}_{H^1,N}$ \\
						\hline
						2 & 10  & 4.17E+3 & 7.34E-3 & &  & 1.10E-1  &  & &  55 & 1.47E+5 &  3.96E-5 & &  & 8.95E-3 &  & & & & & & & & & \\
						3 & 21 & 3.87E+4 & 2.49E-3 & 1.56 & 1.46 & 6.33E-2  & 0.79 & 0.73 & 95 & 1.51E+5 & 3.78E-5  & 0.07 & 0.09 & 8.93E-3 & 2.72E-3 & 3.45E-3 & 256 & 1.65E+5 & 1.56E-3 & & & 4.67E-2 & &  \\
						4 & 40 & 8.29E+4 & 1.19E-4 & 4.39 & 4.72 & 1.20E-2  & 2.40 & 2.58 &  151 & 1.53E+5 &   3.92E-6 & 3.27 & 4.89 & 2.33E-3 & 1.94 & 2.90 & 368 & 1.84E+5 & 9.58E-9 & 17.31 & 33.1 & 1.42E-4 & 8.36 & 16.0 \\
						5 & 75 & 9.56E+4 & 6.43E-5 & 0.88 & 0.97 & 9.73E-3  & 0.30 & 0.33 &  239 & 2.03E+5 &  9.97E-8 & 5.30 & 8.00 & 4.49E-4 & 2.38 & 3.59  & 512 & 1.89E+5 & 8.67E-11 & 6.79 & 14.2 & 1.05E-5 & 3.75 & 7.87 \\
						6 & 142 & 1.01E+5 & 5.09E-5 & 0.34 & 0.37  &  9.01E-3 & 0.11 & 0.12  & 391 & 2.08E+5 & 7.07E-9  & 3.82 & 5.38 & 1.13E-4 & 1.99 & 2.80 & 720 & 1.90E+5 & 6.31E-12 &  3.78 & 7.69 & 3.17E-6 & 1.73 & 3.52 \\
						7 & 273 & 1.05E+5 & 1.24E-5 & 2.03 & 2.16 & 4.44E-3  & 1.02 & 1.08 & 671 & 2.08E+5 &  6.53E-10 & 3.43 & 4.41 & 3.13E-5 & 1.85 & 2.38 & & &  &  &  & & &  \\
						8 & 532 & 1.13E+5 & 2.79E-6 & 2.16 & 2.24 & 2.10E-3  & 1.08 & 1.12 &  1207 & 2.09E+5 &  6.67E-11 & 3.29 & 3.89 & 7.66E-6 & 2.03 & 2.40 & & & & & & & & \\
						9 & 1047 & 1.24E+5 & 3.78E-7 & 2.88 & 2.95 & 7.60E-4  & 1.47 & 1.50 && &   & &  &  & & & & & & & & & & \\
						10 & 2074 & 1.30E+5 & 4.92E-8 & 2.94 & 2.98 & 2.62E-4  & 1.54 & 1.56 & & &   &  &  &  & & & & & & & & & & \\
						\hline
					\end{tabular}%
				}
			\end{center}
			\caption{Numerical results for \cref{exgamma}.} %The slopes are computed by performing a linear regression on $\log_2(\|u_J-u\|_2)$ and $\log_2(\|u'_J-u'\|_2)$ versus $\log_2(N_J)$.}
			\label{ex:gamma}
		\end{table}
	\end{example}
	
	\begin{example} \label{exunknown}
		\normalfont
		Consider the model problem \eqref{model} with $a_- = 1$, $a_+ = 1000 e^x$, $\Gamma:=\pi/6$, $g_\Gamma = 0$, and $f=-1$. The true solution satisfies 
		\[
		u_- = \tfrac{x^2}{2}+cx, \quad
		u_+ = \tfrac{-(x+1+c)e^{-x}+(2+c)e^{-1}}{1000}, \quad 
		c:=-\tfrac{500\Gamma^2+(\Gamma+1)e^{-\Gamma}-2e^{-1}}
		{1000\Gamma+e^{-\Gamma}-e^{-1}}.
		\]
		See \cref{ex:unknown} for the numerical experiments.
		%
		%For the wavelet \cite[Example 3.1]{HM25} with approximation order $2$, the reference solution $u_{\text{ref}} \in \text{span} (\mathcal{B}^{S,H^1_0(\Omega)}_{2,10})$. For the multiwavelet \cite[Example 3.4]{HM25} with approximation order $3$, we use the reference solution $u_{\text{ref}} \in \text{span} (\mathcal{B}^{S,H^1_0(\Omega)}_{2,8})$. For the multiwavelet \cite[Example 3.5]{HM25} with approximation order $4$, we use the reference solution $u_{\text{ref}} \in \text{span} (\mathcal{B}^{S,H^1_0(\Omega)}_{3,6})$. See \cref{ex:unknown} for the numerical experiments.
		%
		\begin{table}[htbp]
			\begin{center}
				\resizebox{\textwidth}{!}{%
					\begin{tabular}{c | c c c c c c c c | c c c c c c c c | c c c c c c c c }
						\hline
						& \multicolumn{8}{c}{$\mathcal{B}^{S,H^1_0(\Omega)}_{2,J}$ using \cite[Example 3.1]{HM25} with approximation order $2$} \vline & \multicolumn{8}{c}{$\mathcal{B}^{S,H^1_0(\Omega)}_{2,J}$ using \cite[Example 3.4]{HM25} with approximation order $3$} \vline &
						\multicolumn{8}{c}{$\mathcal{B}^{S,H^1_0(\Omega)}_{3,J}$ using \cite[Example 3.5]{HM25} with approximation order $4$}  \\
						\hline
						$J$ & $N_{J}$ & $\kappa$ & $\|u_{J}-u\|_2$ & $\text{Ord}_{L^2,h}$ & $\text{Ord}_{L^2,N}$ & $\| u'_J -  u'\|_2 $ & $\text{Ord}_{H^1,h}$ & $\text{Ord}_{H^1,N}$ & $N_{J}$ & $\kappa$ & $\|u_{J}-u\|_2$ & $\text{Ord}_{L^2,h}$ & $\text{Ord}_{L^2,N}$ & $\|u'_J - u'\|_2 $ & $\text{Ord}_{H^1,h}$ & $\text{Ord}_{H^1,N}$ & $N_{J}$ & $\kappa$ & $\|u_{J}-u\|_2$ & $\text{Ord}_{L^2,h}$ & $\text{Ord}_{L^2,N}$ & $\|u_J' - u'\|_2 $ & $\text{Ord}_{H^1,h}$ & $\text{Ord}_{H^1,N}$ \\
						\hline
						2 & 10 & 3.96E+2 & 3.07E-3 & & & 4.55E-2 & & & 55 & 9.05E+3 & 1.76E-5 & & & 3.32E-3 & & & & & & & & & & \\
						3 & 21 & 3.25E+3 & 1.02E-3 & 1.58 & 1.48 & 2.61E-2 & 0.80 & 0.75 & 95 & 9.54E+3 & 1.71E-5 & 0.04 & 0.06 & 3.22E-3 & 0.04 & 0.06 & 256 & 1.09E+4 & 1.62E-7 & & & 3.17E-4 & & \\
						4 & 40 & 7.00E+3 & 7.45E-5 & 3.78 & 4.07 & 7.06E-3 & 1.89 & 2.03 & 151 & 1.00E+4 & 9.35E-7 & 4.19 & 6.27 & 7.60E-4 & 2.08 & 3.11 & 368 & 1.09E+4 & 2.20E-9 & 6.21 & 11.85 & 3.48E-5 & 3.19 & 6.08 \\
						5 & 75 & 8.09E+3 & 3.12E-5 & 1.25 & 1.38 & 4.61E-3 & 0.61 & 0.68 & 239 & 1.20E+4 & 1.14E-8 & 6.36 & 9.60 & 8.45E-5 & 3.17 & 4.78 & 512 & 1.36E+4 & 4.66E-11 & 5.56 & 11.67 & 5.32E-6 & 2.71 & 5.69 \\
						6 & 142 & 8.57E+3 & 2.08E-5 & 0.58 & 0.63 & 3.69E-3 & 0.32 & 0.35 & 391 & 1.22E+4 & 4.50E-9 & 1.34 & 1.89 & 5.08E-5 & 0.73 & 1.03 & 720 & 1.41E+4 & 5.57E-13 & 6.39 & 12.99 & 5.32E-7 & 3.32 & 6.76 \\
						7 & 273 & 8.88E+3 & 5.10E-6 & 2.03 & 2.15 & 1.83E-3 & 1.02 & 1.08 & 671 & 1.30E+4 & 3.60E-11 & 6.96 & 8.94 & 4.75E-6 & 3.42 & 4.39 & & & & & & & & \\
						8 & 532 & 9.34E+3 & 1.16E-6 & 2.14 & 2.22 & 8.71E-4 & 1.07 & 1.11 & 1207 & 1.30E+4 & 1.06E-11 & 1.77 & 2.09 & 2.58E-6 & 0.88 & 1.04 & & & & & & & & \\
						9 & 1047 & 1.01E+4 & 1.70E-7 & 2.77 & 2.83 & 3.37E-4 & 1.37 & 1.40 & & & & & & & & & & & & & & & \\
						10 & 2074 & 1.06E+4 & 2.57E-8 & 2.72 & 2.76 & 1.32E-4 & 1.35 & 1.37 & & & & & & & & & & & & & & & \\
						\hline
					\end{tabular}%
				}
			\end{center}
			\caption{Numerical results for \cref{exunknown}.} %The slopes are computed by performing a linear regression on $\log_2(\|u_J-u\|_2)$ and $\log_2(\|u'_J-u'\|_2)$ versus $\log_2(N_J)$.}
		\label{ex:unknown}
		\end{table}
	\end{example}
		
		\textbf{Remarks for Examples 1-5.} The first halves of \cref{ex:nogamma,tab:threebp} show the effect of ignoring the interface. Without any modifications, the space spanned by $\mathcal{B}^{H^1_0(\Omega)}_{J_0,J}$ defined in \eqref{BH1J0J} is the same as the space spanned by the standard Galerkin/FEM method at the same scale level or mesh size. Thus, low convergence rates are observed. In contrast, the second halves of the table present the results obtained using the augmented finite subset in \eqref{BSH1J0J} and exhibit the following behaviour. When we use an order 2 biorthogonal wavelet, the convergence rates in the $L^2(\Omega)$-norm and in the $H^{1}(\Omega)$-semi-norm are around 2 and 1 respectively. For an order 3 biorthogonal multiwavelet, the rates exceed 3 and 2, respectively; the same trend holds for order 4. Although $N_{J+1}/N_J < 2$ for the tested $J$, \eqref{BSH1J0J} implies $N_{J+1}/N_J \rightarrow 2$ as $J \rightarrow \infty$. As $J$ gets larger, the convergence rates are expected to approach $m$ in the $L^2(\Omega)$-norm and $m-1$ in the $H^{1}(\Omega)$-seminorm, where $m$ is the approximation order. Large values of $J$ make computation intractable due to finite-precision limits in MATLAB.
		
		\begin{example} \label{ex:2D}
			\normalfont 
			Consider the model problem \eqref{model:0} with $a(x,y) = \chi_{[0,1/3) \cup [2/3,1]}(x) + 10^5 \chi_{[1/3,2/3)}(x)$, $\Gamma:=(\{1/3\} \times (0,1)) \cup (\{2/3\} \times (0,1))$, and $f$ is chosen such that the true solution satisfies
			 \[
			 u = \sin(\pi y) (\tfrac{1}{2}(x^2-x) \chi_{[0,1/3) \cup [2/3,1]}(x) + (-\tfrac{1}{9} + 10^{-5} (\tfrac{1}{2}(x^2-x)+\tfrac{1}{9}))\chi_{[1/3,2/3)}(x) ).
			 \]
			For this problem, $g_\Gamma = 0$. Since this is a 2D problem, we take the 2D Riesz basis of the Sobolev space $H^1_0(\Omega)$ in \cite[(2.7)-(2.8)]{HM24}. I.e., we take 
			\[
			 \mathcal{B}^{H^1_0(\Omega)}_{J_0} := [2^{-J_0}\Phi^{2D}_{J_0}] \cup \cup_{j=J_0}^{\infty} [2^{-j}\Psi^{2D}_{j}],
			\]
			where
			\begin{align*}
				& \Phi^{2D}_{J_0} =\Phi_{J_0}\otimes \Phi_{J_0}, \quad \Psi^{2D}_{j} =\{\Phi_j \otimes \Psi_j,  \Psi_{j}  \otimes \Phi_{j}, \Psi_j \otimes \Psi_j\}, \\
				& \tilde{\Phi}^{2D}_{J_0} =\tilde{\Phi}_{J_0}\otimes \tilde{\Phi}_{J_0}, \quad \tilde{\Psi}^{2D}_{j} =\{\tilde{\Phi}_j \otimes \tilde{\Psi}_j,  \tilde{\Psi}_{j}  \otimes \tilde{\Phi}_{j}, \tilde{\Psi}_j \otimes \tilde{\Psi}_j\}.
			\end{align*}
			Take the following finite subset to be used in the Galerkin formulation
			\be \label{BH1:2D}
			\mathcal{B}^{H^1_0(\Omega)}_{J_0,J} := [2^{-J_0}\Phi^{2D}_{J_0}] \cup \cup_{j=J_0}^{J} [2^{-j} \Psi^{2D}_{j}],
			\ee
			and its augmented version 
			\be  \label{BSH1:2D}
			\begin{aligned}
			& \mathcal{B}^{S,H^1_0(\Omega)}_{J_0,J} := [2^{-J_0}\Phi^{2D}_{J_0}] \cup \cup_{j=J_0}^{J-1} [2^{-j} \Psi^{2D}_{j}] \cup [2^{-J} \Psi^{S,2D}_{J}], \quad \text{where}\\
			& \Psi^{S,2D}_{J} := \{\Phi_J \otimes \Psi_J,  \breve{\Psi}_{J}  \otimes \Phi_{J}, \breve\Psi_J \otimes \Psi_J\} \quad \text{and} \quad  \breve{\Psi}_{J}:= \Psi_J \cup \cup_{j=J+1}^{(2m-2)J-1} [2^{J-j}\mathcal{S}_j]
 			\end{aligned}
 			\ee
			with $\mathcal{S}_j$ defined in \eqref{BSH1J0J}. See \cref{tab:2D} for the numerical experiments. 
			\begin{table}[htbp]
				\begin{center}
					\resizebox{\textwidth}{!}{%
						\begin{tabular}{c | c c c c c c c | c c c c c c c}
							\hline
							& \multicolumn{7}{c}{$\mathcal{B}^{S,H^1_0(\Omega)}_{3,J}$ in \eqref{BSH1:2D} using \cite[Example 3.5]{HM25} with approximation order $4$}
							\vline & \multicolumn{7}{c}{FEM version $\mathcal{B}^{H^1_0(\Omega)}_{3,J}$ in \eqref{BH1:2D} using \cite[Example 3.5]{HM25} with approximation order $4$}  \\
							\hline
							$J$
							& $N_J$ & $\mathrm{E}_{L^2}$ & $\mathrm{Ord}_{L^2,h}$ & $\mathrm{Ord}_{L^2,N}$ & $\mathrm{E}_{H^1}$ & $\mathrm{Ord}_{H^1,h}$ & $\mathrm{Ord}_{H^1,N}$
							& $N_J$ & $\mathrm{E}_{L^2}$ & $\mathrm{Ord}_{L^2,h}$ & $\mathrm{Ord}_{L^2,N}$ & $\mathrm{E}_{H^1}$ & $\mathrm{Ord}_{H^1,h}$ & $\mathrm{Ord}_{H^1,N}$ \\
							\hline
							3 &  15168  & 1.22E-07 &      &      & 2.19E-04 &      &
							&1024  & 1.62E-03 &      &      & 3.20E-02 &      &      \\
							4 &  43008  & 5.39E-09 & 4.50 & 5.98 & 2.73E-05 & 3.00 & 3.99
							& 4096  & 1.35E-03 & 0.27 & 0.27 & 2.86E-02 & 0.16 & 0.16 \\
							5 & 114688 & 3.46E-10 & 3.96 & 5.60 & 3.42E-06 & 3.00 & 4.24
							& 16384 & 3.76E-04 & 1.84 & 1.84 & 1.45E-02 & 0.98 & 0.98 \\
							6 &  303104 & 2.04E-11 & 4.09 & 5.83 & 4.27E-07 & 3.00 & 4.28
							&  65536 & 3.17E-04 & 0.25 & 0.25 & 1.31E-02 & 0.14 & 0.14 \\
							\hline
						\end{tabular}%
					}
				\end{center}
				\caption{Numerical results for \cref{ex:2D}.}
				\label{tab:2D}
			\end{table}

	\end{example}

	\textbf{Remarks for Example 6.} Since the interface curve $\Gamma$ in \cref{ex:2D} consists of two vertical lines, we can simply enrich the 2D wavelet basis only in the $x$-direction as described in \eqref{BSH1:2D}. The resulting 2D enriched wavelet basis is simpler than the one used in \cite{HM24}, which deals with general one-dimensional interface curves.  \cref{tab:2D} once again shows that when the singularities along the vertical interface lines are neglected, we observe reduced convergence rates. Our 2D enriched wavelet basis in \eqref{BSH1:2D} successfully resolves these singularities. It is also due to the simplicity of these vertical interface lines that we are able to achieve convergence rates of order 3 in the $H^1(\Omega)$-norm.
	
\section{Proof of \cref{thm:main}} \label{sec:proof}

Because the interface $\Gamma$ is a finite subset of the domain $\Omega$ and $-(a u')'=f-g_\Gamma \delta_\Gamma$ in \eqref{model} is a linear differential equation, it suffices to prove \cref{thm:main} for $\Gamma$ being a singleton. For simplicity, in the rest of this section, we shall use $\Gamma$ itself to denote an interface point in $\Omega$ and we define subdomains
\[
\Omega_-:=(0, \Gamma),\qquad \Omega_+:=(\Gamma,1).
\]
To prove our main result \cref{thm:main},	 we first present a proposition, which will be used in the proof of the $L^{2}(\Omega)$ convergence.
It is a standard result (e.g., see \cite[Lemma 3.7]{LAH09}), but we choose to include it to make the paper self-contained.
	\begin{prop}
		Consider the model problem \eqref{model}, where the source term $f \in L^2(\Omega)$, $g_\Gamma =0$, and the variable diffusion coefficient $a \in L^{\infty}(\Omega)$ satisfies $\text{ess-inf}_{x\in \Omega}(a(x)) > 0$, $a'_- \in L^{\infty}(\Omega_-)$, and $a'_+ \in L^{\infty}(\Omega_+)$. Then, the following bounds on the solution $u$ hold
		\be \label{stab:bound}
		\|u\|_{H^{2}(\Omega_-)} \le C \|f\|_{L^{2}(\Omega)} \quad \text{and} \quad \|u\|_{H^{2}(\Omega_+)} \le C \|f\|_{L^{2}(\Omega)},
		\ee
		where $C$ is a generic constant that only depends on the diffusion coefficient $a$.
	\end{prop}
	\begin{proof}
 		Define $a_{\inf}:=\text{ess-inf}_{x\in \Omega}(a(x))$. The equation in \eqref{model} implies that $u_{-}'' = -a_{-}^{-1} (f_- + a_-' u_-')$ holds in the $L^{2}(\Omega_-)$ sense, and
		\[
		\| u_-''\|_{L^{2}(\Omega_{-})} \le a_{\inf}^{-1} \left( \|f\|_{L^{2}(\Omega)} + \|a_-'\|_{L^{\infty}(\Omega_-)} \|u_-'\|_{L^{2}(\Omega_-)}\right)
		\le C_a ( \|f\|_{L^{2}(\Omega)} +  \|u'\|_{L^{2}(\Omega)} ),
		\]
		where $C_a := a_{\inf}^{-1} (1+ \|a_-'\|_{L^{\infty}(\Omega_-)} )$.
		%
		 %\textcolor{red}{Additionally, by Poincar\'e's inequality, we have $\|u\|_{L^{2}(\Omega)} \le C_{P}  \|u'\|_{L^{2}(\Omega)}$, where $C_{P}$ only depends on $\Omega$.}
		%
		Since $u(0)=0$ in \eqref{model} and $u(x)=\int_0^x u'(t) dt$, we have $|u(x)| \le \|u'\|_{L^2(\Omega)}$ for all $x\in (0,1)$ and
		\begin{equation} \label{poinc:1D}
		\|u\|_{L^2(\Omega)}\le \|u'\|_{L^2(\Omega)}.
		\end{equation}
		
		Note that the weak formulation of \eqref{model} is to find $u$ such that $\la au', v' \ra  =\la f, v\ra$ for all $v \in H^{1}_{0}(\Omega)$.  Taking $v=u$, we have
		$a_{\inf} \| u'\|^2_{L^{2}(\Omega)} \le \|f\|_{L^{2}(\Omega)} \|u\|_{L^{2}(\Omega)}.$
		By \eqref{poinc:1D}, it follows that
		$\| u'\|_{L^{2}(\Omega)} \le a_{\inf}^{-1} \|f\|_{L^{2}(\Omega)}$ and
			$\| u\|_{L^{2}(\Omega)} \le a_{\inf}^{-1} \|f\|_{L^{2}(\Omega)}$.
		Therefore,
		\[
			\|u\|_{H^{2}(\Omega_-)} \le \| u_-'' \|_{L^{2}(\Omega_{-})} + 	\| u'\|_{L^{2}(\Omega)} + 	\| u\|_{L^{2}(\Omega)} \le  C\|f\|_{L^{2}(\Omega)},
		\]
		where $C:=C_a ( 1 + a_{\inf}^{-1}) + 2 a_{\inf}^{-1}$. The estimate for $\|u\|_{H^{2}(\Omega_+)}$ is similarly obtained.
	\end{proof}
	
	The proof of \cref{thm:main} is presented below and involves
	%on the optimal convergence rates of our wavelet Galerkin method with respect to the approximation order of the underlying basis, and the uniform boundedness of the condition number of the coefficient matrix.
	%As can be seen from the discussion below, the proof involves an
	in-depth analysis of the dual wavelets.

	\begin{proof}[Proof of \cref{thm:main}]
	We shall estimate the magnitude of $\la u, 2^j \tilde{\eta}_j\ra$ for $\tilde{\eta}_j \in \tilde{\Psi}_j$ in the wavelet representation \eqref{expr}.
	Recall that
	$f_{j;k}(x) := 2^{j/2} f(2^j x -k)$ for scaled and translated versions of a function $f$, and $\mbox{Supp}(f)$ is the smallest closed interval such that $f$ vanishes outside it. By the definition of $\mathcal{B}^{H^1_0(\Omega)}_{J_0}$ in
	\eqref{BH1J0}, each set $\tilde{\Psi}_j$ is generated by scaling and translating three types of basic wavelets (see the discussion preceding \eqref{BH1J0} or \cite{HM21}):
	interior dual wavelets $\tilde{\psi}$ on $\R$, left boundary dual wavelets $\tilde{\psi}^L$ on $[0,\infty)$, and right boundary dual wavelets $\tilde{\psi}^R$ on $(-\infty,1]$.  For interior dual wavelets, we define
		\[
	\tilde{\psi}^{[0]}(x) := \tilde{\psi}(x) \quad \text{and} \quad \tilde{\psi}^{[n]}(x) := \int_{-\infty}^{x} \tilde{\psi}^{[n-1]}(t) dt, \quad
	x\in \R, n\in \N.
		\]
	For left boundary dual wavelets $\tilde{\psi}^L$ and right boundary dual wavelets $\tilde{\psi}^R$,  we define
	\begin{align*}
	& \tilde{\psi}^{L,[0]}(x) :=
	\tilde{\psi}^L(x), \quad  \tilde{\psi}^{L,[n]}(x) := -\int_{x}^{\infty} \tilde{\psi}^{L,[n-1]}(t) dt, \quad x\in [0,\infty),\; n \in \N, \\
	& \tilde{\psi}^{R,[0]}(x) := \tilde{\psi}^R(x), \quad \tilde{\psi}^{R,[n]}(x) := \int_{-\infty}^{x} \tilde{\psi}^{R,[n-1]}(t) dt, \quad x\in (-\infty, 1],\; n \in \N.
	\end{align*}
Obviously, $(\tilde{\psi}^{[n]})'=\tilde{\psi}^{[n-1]}$,
$(\tilde{\psi}^{L,[n]})'=\tilde{\psi}^{L,[n-1]}$, and
$(\tilde{\psi}^{R,[n]})'=\tilde{\psi}^{R,[n-1]}$ for all $n\in \N$.
Because $\tilde{\psi}$ on $\R$ has compact support, applying the conditions of vanishing moments in \eqref{vm:psi}, we can deduce from the definition of the continuous functions $\tilde{\psi}^{[n]}, n\in \N$ that
\begin{equation}\label{psi:n}
\mbox{Supp}(\tilde{\psi}^{[n]})\subseteq
\mbox{Supp}(\tilde{\psi}), \qquad n=1,\ldots, m.
\end{equation}
For simplicity of discussion, without loss of generality, we shall assume $0\in \mbox{Supp}(\tilde{\psi}^L)$ and $1\in \mbox{Supp}(\tilde{\psi}^R)$. Hence, we have
$\mbox{Supp}(\tilde{\psi}^L)=[0,c]$ and
$\mbox{Supp}(\tilde{\psi}^R)=[d,1]$ with $c\ge 0$ and $d\le 1$.
By the definition of $\tilde{\psi}^{L,[n]}$ and $\tilde{\psi}^{R,[n]}$ for $n\in  \N$, we trivially conclude that all such continuous and differentiable functions satisfy
%\[
%\tilde{\psi}^{L,[n]}(x)=0, \quad
%\tilde{\psi}^{R,[n]}(y)=0, \quad \forall\; x\ge %c, y\le d, n=1,\ldots,m-1,\]
%which just imply
%
\begin{equation}\label{LR:supp}
\mbox{Supp}(\tilde{\psi}^{L,[n]})\subseteq [0,c],\quad
\mbox{Supp}(\tilde{\psi}^{R,[n]})\subseteq [d,1],\quad n=1,\ldots, m.
\end{equation}
Similarly, applying the conditions in \eqref{vm:LR} and using integration by parts, we have
\begin{equation}\label{zeroval}
\tilde{\psi}^{L,[n]}(0)=0,\quad
\tilde{\psi}^{R,[n]}(1)=0,\quad  n=2,\ldots,m.
\end{equation}
However, $\tilde{\psi}^{L,[1]}(0)\ne 0$ and $\tilde{\psi}^{R,[1]}(1)\ne 0$ can happen.
Here we prove \eqref{zeroval} for $\tilde{\psi}^L$. The claims for $\tilde{\psi}$ in \eqref{psi:n} and $\tilde{\psi}^R$ in \eqref{zeroval} are similar.
Note that \eqref{LR:supp} trivially implies
\[
x^j \tilde{\psi}^{L,[n]}(x)|_{x=0}=0,\qquad
x^j \tilde{\psi}^{L,[n]}(x)|_{x=c}=0,\qquad \mbox{for all } j, n\in \N.
\]
Repeatedly applying integration by parts and removing the boundary values by the above identities, for $n=2,\ldots m$, we conclude the identities \eqref{zeroval} for $\tilde{\psi}^L$ from
\begin{align*}
\tilde{\psi}^{L,[n]}(0)
&:=-\int_0^\infty \tilde{\psi}^{L, [n-1]}(t) dt
= -\int_0^c \tilde{\psi}^{L, [n-1]}(t) dt
=\int_0^c t \tilde{\psi}^{L,[n-2]}(t)dt\\
&=\cdots
=\frac{(-1)^{n-1} }{(n-1)!} \int_0^c t^{n-1} \tilde{\psi}^{L,[0]} (t) dt=
\frac{(-1)^{n-1}}{(n-1)!} \int_0^\infty t^{n-1} \tilde{\psi}^{L} (t) dt=
0,
\end{align*}
where we used the assumption \eqref{vm:LR} and
the trivial fact that $n-1\ge 1$ for all $n=2,\ldots,m$.

In what follows, for $\tilde{\eta}_j \in \tilde{\Psi}_j$, we shall consider three cases: $\mbox{Supp}(2^{j} \tilde{\eta}_j) \subseteq \{0\}\cup \Omega_-$, $\mbox{Supp}(2^{j} \tilde{\eta}_j) \subseteq \Omega_+\cup \{1\}$, and $\Gamma \in \mbox{Supp}(2^{j} \tilde{\eta}_j)$. It is clear that $\mbox{Supp}(2^{j} \tilde{\eta}_j) = \mbox{Supp}(\tilde{\eta}_j)$. Also, by change of variables, we have $\| {\tilde{\psi}}^{[m]}_{j;k}\|_{L^2(\Omega)} = \| {\tilde{\psi}}^{[m]}\|_{L^2(\text{Supp}(\tilde{\psi}))}$,  $\| {\tilde{\psi}}^{L,[m]}_{j;0}\|_{L^2(\Omega)} = \| {\tilde{\psi}}^{L,[m]}\|_{L^2(\text{Supp}(\tilde{\psi}^{L}))}$, and  $\| {\tilde{\psi}}^{R,[m]}_{j;2^j-1}\|_{L^2(\Omega)} = \| {\tilde{\psi}}^{R,[m]}\|_{L^2(\text{Supp}(\tilde{\psi}^R))}$.

First, we consider the first case $\mbox{Supp}(2^{j} \tilde{\eta}_j) \subseteq \{0\}\cup \Omega_-$. Then we have
two subcases: $\tilde{\eta}_j = \tilde{\psi}_{j;k}$ for some $k \in \Z$ or $\tilde{\eta}_j = \tilde{\psi}^L_{j;0}$.
Because $\Omega_-$ is open, we must have $\mbox{Supp}(2^{j} \tilde{\eta}_{j})\subseteq [0,\Gamma_-]$ with $\Gamma_-<\Gamma$.
Recall that $u_-= u$ on $\Omega_-$ and $u_- \in H^{m}(\Omega_-)$.
We first consider the subcase $\tilde{\eta}_j = \tilde{\psi}_{j;k}$.  Note that \eqref{psi:n} trivially implies $\tilde{\psi}^{[n]}_{j;k}(0)=
\tilde{\psi}^{[n]}_{j;k}(\Gamma_-)=0$ for all $n\in \N$. Now using integration by parts and noting that the boundary values must be zero, we have
	\begin{equation}\label{tpsi}
		\begin{aligned}
			\la u, 2^j \tilde{\psi}_{j;k} \ra
			& :=\int_0^{\Gamma_-} u(x) 2^j \tilde{\psi}_{j;k}(x) dx = u(x) \tilde{\psi}^{[1]}_{j;k}(x)|_{x=0}^{x=\Gamma_-}
- 2^{-j} \int_0^{\Gamma_-} u'(x) 2^j \tilde{\psi}^{[1]}_{j;k}(x) dx \\
			& = - \int_0^{\Gamma_-} u'(x) \tilde{\psi}^{[1]}_{j;k}(x) dx = \ldots
			%= (-1)^m 2^{-mj} \int_0^{\Gamma} u^{(m)}(x) 2^j {\tilde{\psi}}^{[m]}_{j;k}(x) dx \\
			%&
			= (-1)^m 2^{(1-m)j} \int_{\text{Supp}(\tilde{\psi}_{j;k})} u^{(m)}(x) \tilde{\psi}^{[m]}_{j;k}(x) dx.
		\end{aligned}	
	\end{equation}
	By \eqref{tpsi} and the Cauchy-Schwarz inequality, we have
	\begin{equation} \label{upsijk:smooth}
	|\la u, 2^j \tilde{\psi}_{j;k} \ra|
	%= 2^{(1-m)j} |\la u^{(m)}, {\tilde{\psi}}^{[m]}_{j;k} \ra|
	\le 2^{(1-m)j} \|u^{(m)}\|_{L^{2}(\text{Supp}(\tilde{\psi}_{j;k}))} \|{\tilde{\psi}}^{[m]}_{j;k} \|_{L^{2}(\text{Supp}(\tilde{\psi}_{j;k}))}
	\le C_1 2^{(1-m)j} \|u^{(m)}\|_{L^{2}(\text{Supp}(\tilde{\psi}_{j;k}))},
	\end{equation}
where the constant $C_1$ is defined by
	\[
	C_1 := \max\left\{
	\| {\tilde{\psi}}^{L,[m]}\|_{L^2(\text{Supp}(\tilde{\psi}^{L}))},
	\| {\tilde{\psi}}^{[m]}\|_{L^2(\text{Supp}(\tilde{\psi}))},
	 \| {\tilde{\psi}}^{R,[m]}\|_{L^2(\text{Supp}(\tilde{\psi}^R))}
	\right\}.
	\]
The subcase $\tilde{\eta}_j = \tilde{\psi}^L_{j;0}$ is similar. By $\tilde{\psi}^{L,[1]}_{j;0}(\Gamma_-)=0$ in \eqref{LR:supp} and the assumption $u(0)=0$, we have
\[
	\la u, 2^j \tilde{\psi}^L_{j;0} \ra
	:=\int_0^{\Gamma_-} u(x) 2^j \tilde{\psi}^L_{j;0}(x) dx
 =u(x)\tilde{\psi}^{L,[1]}_{j;0}(x)|_{x=0}^{x=\Gamma_-}
	-\int_0^{\Gamma_-} u'(x) \tilde{\psi}^{L,[1]}_{j;0}(x) dx
=-\int_0^{\Gamma_-} u'(x) \tilde{\psi}^{L,[1]}_{j;0}(x) dx
\]
and by repeatedly applying integration by parts and using \eqref{LR:supp} and \eqref{zeroval}, we have
	\[
	\la u, 2^j \tilde{\psi}^L_{j;0} \ra
=
- \int_0^{\Gamma_-} u'(x) \tilde{\psi}^{L,[1]}_{j;0}(x) dx
	=\cdots
	= (-1)^m 2^{(1-m)j} \int_{\text{Supp}(\tilde{\psi}^L_{j;0})} u^{(m)}(x) \tilde{\psi}^{L,[m]}_{j;0}(x) dx.
	\]
	By the same argument as in \eqref{upsijk:smooth}, we have
	\be \label{upsijk:smooth:L}
	|\la u, 2^j \tilde{\psi}^L_{j;0}\ra|
	\le C_1 2^{(1-m)j} \| u^{(m)}\|_{L^2(\text{Supp}(\tilde{\psi}^L_{j;0}))}.
	\ee
	
	The second case $\mbox{Supp}(2^{j} \tilde{\eta}_j) \subseteq \Omega_+\cup\{1\}$ with $\tilde{\eta}_j \in \tilde{\Psi}_j$ is almost identical to the first case. We just change $u_-$, $u(0)=0$ and $[0,\Gamma_-]$ to $u_+$, $u(1)=0$, and $[\Gamma_+,1]$, respectively, where $\mbox{Supp}(2^{j} \tilde{\eta}_j)\subseteq [\Gamma_+,1]$ and $\Gamma<\Gamma_+\le 1$. Consequently, \eqref{upsijk:smooth} still holds and \eqref{upsijk:smooth:L} becomes
	\be \label{upsijk:smooth:R}
	|\la u, 2^j \tilde{\psi}^R_{j;2^j-1}\ra|
	\le C_1 2^{(1-m)j} \| u^{(m)}\|_{L^2(\text{Supp}(\tilde{\psi}^R_{j;2^j-1}))}.
	\ee

Note that $\Gamma$ is an interior point of the open domain $\Omega$ and the supports of $\tilde{\psi}^L_{j;0}$ and $\tilde{\psi}^R_{j;2^j-1}$ approach $0$ and $1$ respectively as $j$ gets larger. As our last case, it is enough to consider $\tilde{\eta}_j = \tilde{\psi}_{j;k}$, where $\Gamma \in \mbox{Supp}(2^{j} \tilde{\eta}_j)$ and $\mbox{Supp}(2^{j} \tilde{\eta}_j)\subseteq \Omega$.
%Because $\Gamma$ is an interior point of the open domain $\Omega$, we must have $\tilde{\eta}_j=2^j \tilde{\psi}_{j;k}$ and $\mbox{Supp}(\tilde{\psi}_{j;k})\subseteq \Omega$.
%	Then, we have
%	%
%	\begin{align}
%	\nonumber
%	& \la u, 2^j \tilde{\psi}_{j;k} \ra
%	:= \int_0^\Gamma u(x) 2^j \tilde{\psi}_{j;k}(x) dx + \int_\Gamma^1 u(x) 2^j \tilde{\psi}_{j;k}(x) dx \\
%	\nonumber
%	& \quad = (u(\Gamma + ) - u(\Gamma -)){\tilde{\psi}}^{[1]}_{j;k}(\Gamma) - u(0) {\tilde{\psi}}^{[1]}_{j;k}(0) - \int_0^\Gamma  u'(x) {\tilde{\psi}}^{[1]}_{j;k}(x) dx
%	+ u(1) {\tilde{\psi}}^{[1]}_{j;k}(1) - \int_\Gamma^1  u'(x) {\tilde{\psi}}^{[1]}_{j;k}(x) dx\\
%	\label{upsijk:gam}
%	& \quad = - \int_{\Omega_+} u'(x) {\tilde{\psi}}^{[1]}_{j;k}(x) dx
%	- \int_{\Omega_-} u'(x) {\tilde{\psi}}^{[1]}_{j;k}(x) dx,
%	\end{align}
	%
Then by $u\in H^1(\Omega)$ with $u(0)=u(1)=0$, we conclude from integration by parts that
\[
\la u, 2^j \tilde{\psi}_{j;k} \ra
=-\int_\Omega u'(x) \tilde{\psi}^{[1]}_{j;k}(x) dx=
-\int_{\Omega_-} u_-'(x) \tilde{\psi}^{[1]}_{j;k}(x) dx
-\int_{\Omega_+} u_+'(x) \tilde{\psi}^{[1]}_{j;k}(x) dx.
\]
Since $u_+ \in H^{m}(\Omega_+)$ and $u_- \in H^{m}(\Omega_-)$ with $m \ge 2$, we trivially have $u_+ \in H^{2}(\Omega_+)$ and $u_- \in H^{2}(\Omega_-)$.
Because $u_-' \in H^{1}(\Omega_-)\subset C(\overline{\Omega_-})$ and
$u_+' \in H^{1}(\Omega_+)\subset C(\overline{\Omega_+})$ by the Sobolev Embedding Theorem,
there is $C_0>0$ such that $\|u'_-\|_{L^\infty(\Omega_-)}
+\|u'_+\|_{L^\infty(\Omega_+)}\le C_0(\|u_-\|_{H^2(\Omega_-)} +
\|u_+\|_{H^2(\Omega_+)})$ and
	 \begin{equation}\label{upsijk:gamma}
	|\la u, 2^j \tilde{\psi}_{j;k}\ra|
\le (\|u'\|_{L^{\infty}(\Omega_{-})} + \|u'\|_{L^{\infty}(\Omega_{+})}) \|\tilde{\psi}^{[1]}_{j;k}\|_{L^1(\R)}
\le 2^{-j/2} C_2
 (\|u_-\|_{H^{m}(\Omega_{-})} + \|u_+\|_{H^{m}(\Omega_{+})}),
 	\end{equation}
	where $C_2:=C_0 \|\tilde{\psi}^{[1]}\|_{L^1(\R)} < \infty$ and
we have used the fact that $\|\tilde{\psi}^{[m]}_{j;k}\|_{L^1(\R)} = 2^{-j/2} \|\tilde{\psi}^{[m]}\|_{L^1(\R)}$.
	
	Now, for a given level $j$, define
	\[
	\tilde{\mathcal{Y}}_j := \{\tilde{\eta} \in \tilde{\Psi}_j : \Gamma \in \mbox{Supp}(\tilde{\eta})\} \quad \text{and} \quad
	\mathcal{Y}_j := \{\eta \in \Psi_j : \tilde{\eta} \in \tilde{\mathcal{Y}}_j\}.
	\]
	That is, the set $\mathcal{Y}_j$ contains all (primal) wavelets at scale level $j$, whose dual parts touch or overlap with the interface $\Gamma$. It is easy to see that the cardinality of the set $\tilde{\mathcal{Y}}_j$ is $|\tilde{\mathcal{Y}}_j|  = r(\lfloor 2^{j} \Gamma - l_{\tilde{\psi}}\rfloor - \lceil 2^{j} \Gamma - h_{\tilde{\psi}}\rceil + 1)$ and  $|\tilde{\mathcal{Y}}_j| = |\mathcal{Y}_j|$, where $r$ is the multiplicity of $\tilde{\psi}$ and $\text{Supp}(\tilde{\psi}) := [l_{\tilde{\psi}},h_{\tilde{\psi}}]$. Due to the compact support of $\tilde{\psi}$, there is a positive constant $C_\Gamma$, which depends on the wavelet basis and the interface $\Gamma$, such that $|\tilde{\mathcal{Y}}_j| \le C_\Gamma$ for all $j \ge J_0$. It follows from \eqref{upsijk:gamma} that
	\begin{equation} \label{sum:u:psitil:K}
		\begin{aligned}
			 \sum_{j=(2m-2)J}^{\infty} \sum_{\tilde{\eta}_j \in \tilde{\mathcal{Y}}_j} |\la u, 2^j \tilde{\eta}_{j} \ra|^2  & \le C_\Gamma C_2^2 \sum_{j=(2m-2)J}^{\infty}  2^{-j}  \left(\|u\|_{H^{m}(\Omega_+)} + \|u\|_{H^{m}(\Omega_-)}\right)^2 \\
			& \le 2 C_\Gamma C_2^2 2^{-2(m-1)J} \left(\|u\|_{H^{m}(\Omega_+)} + \|u\|_{H^{m}(\Omega_-)}\right)^2.
		\end{aligned}
	\end{equation}
	Define $\tilde{\mathcal{Y}}_j^{c} := \tilde{\Psi}_j \backslash \tilde{\mathcal{Y}}_j$ and $\mathcal{Y}_j^{c} := \Psi_j \backslash \mathcal{Y}_j$. By \eqref{upsijk:smooth} for $\text{Supp}(2^j \tilde{\eta}_{j} )\subseteq \Omega_-$ or $\text{Supp}(2^j \tilde{\eta}_{j} ) \subseteq \Omega_+$, \eqref{upsijk:smooth:L}, and \eqref{upsijk:smooth:R}, we have
	\be \label{sum:u:psitil:Kc}
		\begin{aligned}
			\sum_{j=J+1}^{\infty} \sum_{\tilde{\eta}_j \in \tilde{\mathcal{Y}}_j^c} |\la u, 2^j \tilde{\eta}_{j} \ra|^2
			& \le \sum_{j=J+1}^{\infty} C_1^2 2^{-2(m-1)j} \sum_{\tilde{\eta}_j \in \tilde{\mathcal{Y}}_j^c} \|u^{(m)}\|^2_{L^{2}(\text{Supp}(\tilde{\eta}_j))} \\
			& \le C_1^2 C_{\tilde{\psi}} \sum_{j=J+1}^{\infty}  2^{-2(m-1)j}
			 \left(\|u\|_{H^{m}(\Omega_+)} + \|u\|_{H^{m}(\Omega_-)}\right)^2 \\
			& \le 2 C_1^{2} C_{\tilde{\psi}} 2^{-2(m-1)(J+1)} \left(\|u\|_{H^{m}(\Omega_+)} + \|u\|_{H^{m}(\Omega_-)}\right)^2,
		\end{aligned}
	\ee
	where $C_{\tilde{\psi}}:=\max\{r,\#\tilde{\psi}^L,\#\tilde{\psi}^R\} \max\{|\text{Supp}(\tilde{\psi})|,|\text{Supp}(\tilde{\psi}^L)|,|\text{Supp}(\tilde{\psi}^R)|\}$, $\#\tilde{\psi}^L$, $\#\tilde{\psi}^R$ stand for the number of the left/right boundary dual wavelets, and $|\text{Supp}(\tilde{\psi})|,|\text{Supp}(\tilde{\psi}^L)|,|\text{Supp}(\tilde{\psi}^R)|$ stand for the support lengths of the dual wavelet functions.
		
	Since we assume that $u \in H^{1}_{0}(\Omega)$, the solution $u$ admits the following wavelet representation
	\[
	u = \sum_{\eta \in \Phi_{J_0}} \la u, \tilde{\eta} \ra \eta + \sum_{j  = J_0}^{\infty} \sum_{\eta_j \in \Psi_{j}} \la u, \tilde{\eta}_j \ra \eta_j  =
	\sum_{\eta \in \Phi_{J_0}} \la u, 2^{J_0} \tilde{\eta} \ra 2^{-J_0} \eta + \sum_{j  = J_0}^{\infty} \sum_{\eta_j \in \Psi_{j}} \la u, 2^j \tilde{\eta}_j \ra 2^{-j} \eta_j.
	\]
	By definition of $V^{wav}_{h}$, we have for $\mathring{u}_J \in V^{wav}_h$,
	\[
	u - \mathring{u}_J = I_1 + I_2, \quad
	I_1 := \sum_{j=J+1}^{\infty} \sum_{\eta_j \in \mathcal{Y}_j^c} \la u, 2^j \tilde{\eta}_j \ra 2^{-j} \eta_j \quad
	\text{and}
	\quad
	I_2:=\sum_{j=(2m-2)J}^{\infty} \sum_{\eta_j \in \mathcal{Y}_j}\la u, 2^j \tilde{\eta}_j \ra 2^{-j} \eta_j .
	\]
	Recall that $\mathcal{B}^{H^{1}_0(\Omega)}_{J_0}$ (from which the finite subset $\mathcal{B}^{S,H^{1}_0(\Omega)}_{J_0,J}$ is obtained) is a Riesz basis of $H^{1}_0(\Omega)$. By \eqref{Riesz:stab} (also see \cite[(4.6.9)]{hanbook}), there is a constant $C_3$ such that
	\begin{equation} \label{uuhring}
	\begin{aligned}
	\| u - \mathring{u}_h\|^2_{H^{1}(\Omega)} & = \|I_1 + I_2\|^2_{H^{1}(\Omega)}  \le C_3 \left( \sum_{j=J+1}^{\infty} \sum_{\tilde{\eta}_j \in \tilde{\mathcal{Y}}_j^c} |\la u, 2^j \tilde{\eta}_j \ra|^2 + \sum_{j=(2m-2)J}^{\infty} \sum_{\tilde{\eta}_j \in \tilde{\mathcal{Y}}_j} |\la u, 2^j \tilde{\eta}_j \ra|^2 \right)\\
	& \quad \le C_3 \left(2 C_\Gamma C_2^2 2^{-2(m-1)J} +  2 C_1^{2} C_{\tilde{\psi}} 2^{-2(m-1)(J+1)}\right) \left(\|u\|_{H^{m}(\Omega_+)} + \|u\|_{H^{m}(\Omega_-)}\right)^2\\
	& \quad \le C_4 2^{-2(m-1)J} \left(\|u\|_{H^{m}(\Omega_+)} + \|u\|_{H^{m}(\Omega_-)}\right)^2,
	\end{aligned}
	\end{equation}
	where $C_4 := C_3 \left(2 C_\Gamma C_2^2 +  2 C_1^{2} C_{\tilde{\psi}} 2^{-2(m-1)}\right)$, and we have used \eqref{sum:u:psitil:K}-\eqref{sum:u:psitil:Kc} to obtain the second inequality. Now, C\'{e}a's lemma states that there is a positive constant $C_a$, which depends only on the diffusion coefficient $a$, such that
	\[
	\|u -u_h \|_{H^{1}(\Omega)} \le C_a \inf_{v \in V^{wav}_h} \|u - v\|_{H^{1}(\Omega)}.
	\]
	Combined with \eqref{uuhring}, this implies that
	\begin{equation} \label{uuh:H1:bound}
	\begin{aligned}
	\|u -u_h \|_{H^{1}(\Omega)} & \le C_a \inf_{v \in V^{wav}_h} \|u - v\|_{H^{1}(\Omega)} \le C_a \|u - \mathring{u}_h\|_{H^{1}(\Omega)}\\
	& \le C_a \sqrt{C_4} 2^{-(m-1)J} \left(\|u\|_{H^{m}(\Omega_+)} + \|u\|_{H^{m}(\Omega_-)}\right).
	\end{aligned}
	\end{equation}
	Since $h = 2^{-J}$ and the cardinality of the finite subset used in the approximation at a given scale level $J$ is $N_J  = \mathscr{O}(2^{J})$, we have proved the desired $H^{1}(\Omega)$ convergence.
	
	We use the Aubin-Nitsche's technique to prove the $L^{2}(\Omega)$ convergence. Clearly, the bilinear form $B(u,v):=\la a u' , v' \ra$ is symmetric. Assume that $w \in H^{1}_0(\Omega)$ satisfies
	\[
	B(w,v) = \la u -u_h, v \ra, \quad v \in H^{1}_0(\Omega),
	\]
	while its approximate solution $w_h \in V^{wav}_h$ satisfies
	\[
	B(w_h,v_h) = \la u - u_h, v_h \ra, \quad v_h \in V^{wav}_h.
	\]
	Since the Galerkin orthogonality yields $B(w_h,u-u_h) = B(u - u_h, w_h) =0$ for $w_h \in V^{wav}_h$ and $v = u -u_h \in H^{1}_0(\Omega)$, it follows that
	\begin{equation} \label{uuh:L2:bound}
	\| u  - u_h \|^2_{L^{2}(\Omega)} = B(w,u-u_h) = B(w-w_h,u-u_h) \le C_B \|(w-w_h)'\|_{L^2(\Omega)} \|(u-u_h)'\|_{L^2(\Omega)}
	\end{equation}
	for some positive constant $C_B$, which only depends on $a$. By \eqref{uuh:H1:bound} with $m=2$, we have
	\[
	\|(w-w_h)'\|_{L^2(\Omega)} \le C_a \sqrt{C_4} 2^{-J} (\|w\|_{H^{2}(\Omega_+)} + \|w\|_{H^{2}(\Omega_-)}).
	\]
	Moreover, \eqref{stab:bound} states that there is a constant $C_5$ such that $\|w\|_{H^2(\Omega_-)} \le C_5 \| u - u_h\|_{L^2(\Omega)}$ and $\|w\|_{H^2(\Omega_+)} \le C_5 \| u - u_h\|_{L^2(\Omega)}$. From \eqref{uuh:H1:bound}-\eqref{uuh:L2:bound}, we have
	\[
	\| u  - u_h \|^2_{L^{2}(\Omega)} \le 2 C_B C_a^2 C_4 C_5 2^{-mJ} \| u - u_h\|_{L^2(\Omega)} (\|u\|_{H^{m}(\Omega_+)} + \|u\|_{H^{m}(\Omega_-)}).
	\]
	Cancelling a factor of $\| u  - u_h \|_{L^{2}(\Omega)}$ from both sides, and noting that $h = 2^{-J}$ and $N_J = \mathscr{O}(2^J)$, we obtain the desired conclusion.
	
	To prove the uniform boundedness of the condition number, we can use the same steps as in the proof of \cite[Theorem 2.2]{HM24}.
	\end{proof}	
	
	\section{Statement on the use of AI}
	OpenAI Codex (GPT-5.6 Sol) was only used to assist in adapting author-written MATLAB codes to handle multiple interface points in $\Gamma$ and to implement the 2D numerical example, to execute Examples 3 and 6, to plot Figure 2, and to proofread the manuscript. The authors independently reviewed and verified all AI-assisted outputs and take full responsibility for the entire content of the manuscript.

\end{document}